\documentclass{amsart}

\usepackage[english]{babel}
\usepackage[utf8]{inputenc}
\usepackage[T1]{fontenc}

\usepackage{amssymb,bm,csquotes,enumitem,graphicx,mathrsfs,stmaryrd,tikz}
\usepackage{hyperref}
\usepackage{cleveref}

\def\epsilon{\varepsilon}

\def\NN{\mathbb{N}}
\def\ZZ{\mathbb{Z}}
\def\QQ{\mathbb{Q}}
\def\RR{\mathbb{R}}
\def\CC{\mathbb{C}}

\newcommand{\inttt}[1]{\lvert \kern-0.25ex \lvert \kern-0.25ex \lvert #1 \rvert \kern-0.25ex \rvert \kern-0.25ex \rvert}
\newcommand{\nor}[1]{\lvert \kern-0.25ex \lvert #1 \rvert \kern-0.25ex \rvert}

\newcommand{\matriz}[1]{\begin{pmatrix} #1 \end{pmatrix}}
\def\uu{\bm{u}}

\newenvironment{psmallmatrix}
{\big(\begin{smallmatrix}}
{\end{smallmatrix}\big)}

\def\TT{\operatorname{\mathcal{T}}}

\def\WW{\mathcal{W}}

\def\SS{\mathcal{S}}
\def\SSS{\texorpdfstring{$\mathcal{S}$}{}}
\def\AA{\mathcal{A}}
\def\BB{\mathcal{B}}
\def\LL{\mathcal{L}}
\def\btau{\bm{\tau}}
\def\bmu{\bm{\mu}}
\def\rank{\operatorname{rank}}
\def\Aut{\operatorname{Aut}}
\def\Inf{\operatorname{Inf}}

\def\GG{\mathcal{G}}
\def\HH{\mathcal{H}}
\def\JJ{\mathcal{J}}
\def\diam{\operatorname{diam}}

\def\ima{\operatorname{Im}}
\def\Xeq{X_{\text{eq}}}
\def\Teq{T_{\text{eq}}}
\def\pieq{\pi_{\text{eq}}}

\usepackage[
style=alphabetic,
natbib=true,
backend=bibtex,
giveninits=true,
minalphanames=3,
maxalphanames=4,
maxbibnames=99
]{biblatex}

\newtheorem{theorem}{Theorem}[section]
\newtheorem{proposition}[theorem]{Proposition}
\newtheorem{lemma}[theorem]{Lemma}
\newtheorem{corollary}[theorem]{Corollary}
\newtheorem{claim}{Claim}[theorem]

\theoremstyle{definition}

\newtheorem{example}[theorem]{Example}

\theoremstyle{remark}
\newtheorem{remark}[theorem]{Remark}

\title{Dynamical properties of minimal Ferenczi subshifts}

\author{Felipe Arbulú}
\address{
Laboratoire Amiénois de Mathématique Fondamentale et Apliquée,
CNRS-UMR 7352,
Université de Picardie Jules Verne,
33 rue Saint Leu,
80039 Amiens cedex 1,
France.}
\email{felipe.arbulu@u-picardie.fr}

\author{Fabien Durand}
\address{
Laboratoire Amiénois de Mathématique Fondamentale et Apliquée,
CNRS-UMR 7352,
Université de Picardie Jules Verne,
33 rue Saint Leu,
80039 Amiens cedex 1,
France.}
\email{fabien.durand@u-picardie.fr}

\date{\today}

\begin{document}




\begin{abstract}
We provide an explicit $\SS$-adic representation of rank one subshifts with bounded spacers and call the subshifts obtained in this way ``Ferenczi subshifts''. 
We aim to show that this approach is very convenient to study the dynamical behavior of rank one systems. 
For instance, we compute their topological rank, the strong and the weak orbit equivalence class.
We observe that they have an induced systems that is a Toeplitz subshift having discrete spectrum.
We also characterize continuous and non continuous eigenvalues of minimal Ferenczi subshifts.
\end{abstract}

\maketitle

\section{Introduction}

Cutting and stacking transformations have been used extensively for more than $50$ years in ergodic theory to produce a wide variety of dynamical systems which exhibit different behaviors \cites{Cha67, Orn72, Jun76, Kin86, Kin88, Bou93, Ada98, Nad98, Ryz20, CPR22, Cre22}.
These articles mainly concern the spectral properties, the centralizer and the disjointness of these transformations.

To understand how simple are these systems, in \cite{ORW82} is introduced the notion of (measurable) \emph{rank} to formalize some constructions initiated by Chacon in \cite{Cha67}.
Roughly speaking, the measurable rank is the minimal number of ``stacks'' needed in the cutting and stacking process.
They are defined by two sequences, usually called \emph{cutting} and \emph{spacer} parameters.
The systems needing a unique stack are called \emph{rank one systems} and should be thought as the simplest systems with respect to this notion.
It includes periodic systems and rotations on compact groups \cite{Jun76}, but also many other systems that devoted a lot of attention since the late $60$'s, as ``almost all'' interval exchanges \cites{Fer97, Vee84}.
They have been mainly studied from a spectral and probabilistic point of view, and served to create examples and counterexamples in ergodic theory.
For instance, the Chacon transformation \cite{Cha67} is one of the first known examples of a measurable transformation which is weakly-mixing but not mixing.

S. Ferenczi \cite{Fer96, Fer97} proposed a different perspective representing these systems as subshifts, whereas they have a purely measure theoretic and geometric origin.
This combinatorial and topological model, that can be tracked down to \cite{Kal84}, imposed a different framework and led to many different questions.
For instance, these subshifts are known to have zero topological entropy.
Moreover, they have nonsuperlinear symbolic complexity  \cite[Proposition 2]{Fer96}, but they may have peaks with any prescribed sub exponential growth \cite[Proposition 3]{Fer96}.
We refer to \cites{GH14, GH16a, GH16b, AFP17, GZ19, GZ20, GH21} for recent results about the combinatorial and topological models of rank one systems.

For minimal systems defined on Cantor spaces, there exists a different and well-established notion of rank, called the \emph{topological rank} \cites{DM08, Dur10, BDM10, BKMS13, DP22}.
The class of systems of topological rank one coincides with the class of odometers, so we decided to refer to the symbolic construction of minimal rank one systems as \emph{Ferenczi subshifts} to avoid any misleading definition.
Moreover, S. Ferenczi being the one that popularized this class of subshifts \cites{Fer96, Fer97}, we came naturally to coin his name to them.

This article is devoted to the study of minimal Ferenczi subshifts, \emph{i.e.}, those defined by a uniformly bounded sequence of spacers.
We attempt to create a comprehensive classification for minimal Ferenczi subshifts according to some dynamical properties that we find relevant.
More specifically, we want to compute their topological rank and to describe their (strong and weak) orbit equivalence class, to describe their (continuous and measurable) spectrum, to explore its mixing properties and to compute their automorphism group.

We begin by making the crucial observation that a subshift is a minimal Ferenczi subshift if and only if it is an $\SS$-adic subshift generated by a particular directive sequence of finite alphabet rank.
The family of $\SS$-adic subshifts, introduced by S. Ferenczi in \cite{Fer96}, is a rich family that has been intensively studied and proposed a lot of different behaviours \cites{Dur00, BD14, Ler14, BSTY19, DDMP21}.

It is particularly desirable to have primitive, proper and recognizable directive sequences as this allows, without effort, to define a nested sequence of Kakutani--Rokhlin partitions in towers \cite{DL12}.
This is a central tool for the study of the dynamical properties.
For instance, systems admitting such partitions with a uniform bound for the number of towers are of zero topological entropy \cite{Dur10}, have an explicit description of their ergodic invariant probability measures \cite{BKMS13} and there exist necessary and sufficient conditions for a complex number to be a continuous or measurable eigenvalue \cites{BDM10, DFM19}.

The directive sequence of morphisms we obtain for minimal Ferenczi subshifts has some nice properties, however they are not \emph{proper}.
A recent result of B. Espinoza \cite{Esp22} shows that this directive sequence can be chosen to be proper, but his general method deteriorates the nice structure of the morphisms we obtained and considerably increases the size of the alphabets.
Nevertheless, we can perform a standard trick which guarantees properness, keeping a nice structure of the morphisms and the alphabets.

A direct consequence of the nice structure of the morphisms generating a minimal Ferenczi subshift is that we can compute the topological rank in terms of the cutting and spacer parameters, we recover the well-known fact that they are uniquely ergodic and we show that they have a Toeplitz subshift as an induced system.
Moreover, we show that this induced system is mean equicontinuous and thus has discrete spectrum \cites{LTY15, DG16, GJY21}.

We characterize the \emph{exact finite rank} of the directive sequences for minimal Ferenczi subshifts, \emph{i.e.}, when all towers decomposing the system have a measure bounded away from zero at each level \cite{BKMS13}.
This has an incidence in the study of measurable eigenvalues, as we give a general necessary condition for a complex number to be a measurable eigenvalue for $\SS$-adic subshifts.
We believe this result has its own interest for further studies.
It extends to subshifts what it is often called the \emph{Veech criterion} for interval exchange transformations \cite{Vee84}.

In order to understand the (strong and weak) orbit equivalence class of minimal Ferenczi subshifts and their infinitesimals (in the spirit of \cite{GPS95}), we provide a one-to-one correspondence between the orbit equivalence classes and a family of dimension groups, that we call of \emph{Ferenczi type}.

We then turn to the study of eigenvalues of minimal Ferenczi subshifts.
The group of measurable eigenvalues of a given system gives useful information, as it defines the Kronecker factor that comes naturally with the result of Halmos and von Neumann \cite{HN42}, and also allows to study the weakly-mixing property.
In the topological dynamics counterpart, the group of continuous eigenvalues allows us to understand the maximal equicontinuous factor (in the minimal case) and the topological weakly-mixing property.

In general, it is not true that measurable eigenvalues are continuous.
Measurable eigenvalues coincides with continuous ones for the class of primitive substitution systems \cite{Hos86}.
However, there exist linearly recurrent minimal Cantor systems with measurable and noncontinuous eigenvalues \cite{BDM05}.

In this article, we adopt the general framework of \cites{BDM10, DFM19} to study eigenvalues of minimal Ferenczi subshifts.
This allows to give an alternative proof about the description of continuous eigenvalues \cites{GH16a, GZ19} and
to show that all measurable eigenvalues are continuous in the \emph{exact finite rank} case, which extends a result in \cite{GH16a}.
We also provide some realization results in the non exact finite rank case with noncontinuous eigenvalues.

We also explore the mixing properties of minimal Ferenczi subshifts.
With this purpose, inspired by results in \cite{KSS05}, we give a general necessary condition for topological mixing of minimal subshifts defined on a binary alphabet.
This gives an alternative proof to the fact that minimal Ferenczi subshifts are not topologically mixing \cite{GZ19}.

Finally, we show that subshifts in this family have a unique asymptotic class, which by a standard argument implies that the automorphism group is trivial.
This gives an alternative proof of a result in \cite{GH16b}.

We expect that this $\SS$-adic approach is convenient to investigate some other relevant questions in topological and measurable dynamics of subshifts.

\subsection{Organization}

In the next section we give the basic background in topological dynamics and $\SS$-adic subshifts needed in this article.
We characterize minimal Ferenczi subshifts as those $\SS$-adic subshifts generated by particular directive sequences in \Cref{s:ferenczi_subshifts}.
\Cref{s:top_dyn} is devoted to the study of these subshifts from the topological dynamics viewpoint.
We compute the topological rank and the dimension group of minimal Ferenczi subshifts and their strong and weak orbit equivalence classes.
Then, we study the continuous eigenvalues, the maximal equicontinuous factor and the topological mixing of minimal Ferenczi subshifts.
In the last part of the section, we show that minimal Ferenczi subshifts have a unique asymptotic class and a trivial automorphism group.

We study the measurable eigenvalues of minimal Ferenczi subshifts in \Cref{s:measurable}.
We illustrate these results with concrete examples.

In this article, we let $\NN$ and $\ZZ$ denote the set of nonnegative integers and the set of integers numbers, respectively.
For a finite set $\AA$, we also denote by $\RR_+^\AA$ (resp. $\ZZ_+^\AA$) the set of nonnegative vectors (resp. nonnegative integer vectors) indexed by $\AA$.
Similarly, we denote by $\RR_{>0}^\AA$ (resp. $\ZZ_{>0}^\AA$) to the set of positive vectors (resp. positive integer vectors).
For a vector $v$ in $\RR^\AA$ the Euclidean norm of $v$ is denoted by $\nor{v}$ and we write $\inttt{v} = \inf_{w \in \ZZ^\AA} \nor{v - w}$.

\subsection{Acknowledgments}

The authors are grateful to Bastián Espinoza, Alejandro Maass and Samuel Petite for fruitful discussions.
The first author is supported by ANID (ex CONICYT) Doctorado Becas Chile $72210185$ grant.

\section{Preliminaries}

\subsection{Basics in topological dynamics and eigenvalues}

A \emph{topological dynamical system} (or just a system) is a compact metric space $X$ together with a homeomorphism $T : X \to X$.
We use the notation $(X, T)$.
If $X$ is a Cantor space (\emph{i.e.}, $X$ has a countable basis of clopen sets and it has no isolated points) we say it is a \emph{Cantor system}.
The system $(X, T)$ is \emph{minimal} if for every point $x \in X$ the orbit $\{T^n x : n \in \ZZ\}$ is dense in $X$.

Let $(X, T)$ and $(X', T')$ be two topological dynamical systems.
We say that $(X', T')$ is a \emph{topological factor} of $(X, T)$ if there exists a continuous and surjective map $\phi : X \to X'$ such that
\begin{equation}\label{eq:top_factor}
\phi \circ T
= T' \circ \phi.
\end{equation}
In this case, we say that $\phi$ a \emph{factor map}.
If in addition the map $\phi$ in \eqref{eq:top_factor} is a homeomorphism, we say that it is a \emph{topological conjugacy} and that $(X, T)$ and $(X', T')$ are \emph{topologically conjugate}.

Let $(X, T)$ be a minimal Cantor system and $U \subseteq X$ be a nonempty clopen set.
We can define the \emph{return time function} $r_U : X \to \NN$ by
\[
r_U(x)
= \inf \{n > 0 : T^n x \in U\}, \quad
x \in X.
\]
It is easy to see that the map $r_U$ is locally constant, and hence continuous.
The \emph{induced map} $T_U : U \to U$ is defined by
\[
T_U(x) = T^{r_U(x)} x, \quad
x \in U.
\]
We have that $T_U : U \to U$ is a homeomorphism and that $(U, T_U)$ is a minimal Cantor system.
We call it the \emph{induced system} of $(X, T)$ on $U$.

We say that a complex number $\lambda$ is a \emph{continuous eigenvalue} of the system $(X, T)$ if there exists a continuous function $f : X \to \CC$, $f \not= 0$, such that $f \circ T = \lambda f$; $f$ is called a \emph{continuous eigenfunction} associated with $\lambda$.
The system $(X, T)$ is \emph{topologically weakly-mixing} if it has no nonconstant continuous eigenfunctions.

Let $\mu$ be a $T$-invariant probability measure defined on the Borel $\sigma$-algebra of $X$, \emph{i.e.}, $\mu(T^{-1}(A)) = \mu(A)$ for every measurable set $A \subseteq X$.
We say that a complex number $\lambda$ is a \emph{measurable eigenvalue} of the system $(X, T)$ with respect to $\mu$ if there exists $f \in L^2(X, \mu)$, $f \not= 0$, such that $f \circ T = \lambda f$; $f$ is called a  \emph{measurable eigenfunction} associated with $\lambda$.
The system is \emph{weakly-mixing} for $\mu$ if it has no nonconstant measurable eigenfunctions.

If the system $(X, T)$ is minimal (resp. if $\mu$ is ergodic for $(X, T)$), then every continuous eigenvalue (resp. measurable eigenvalue with respect to $\mu$) has modulus $1$ and every continuous eigenfunction (resp. measurable eigenfunction) has a constant modulus on $X$ (resp. a constant modulus $\mu$-almost everywhere on $X$).

Whenever the measure $\mu$ is ergodic for $(X, T)$ or when $(X, T)$ is minimal, we write $\lambda = \exp(2 \pi i \alpha)$ with $\alpha \in [0,1)$ to denote eigenvalues of the system.
If $\lambda = \exp(2 \pi i \alpha)$ is an eigenvalue of the system with $\alpha$ an irrational number (resp. rational number), we say that $\lambda$ is an irrational eigenvalue (resp. rational eigenvalue).

\subsection{Basics in symbolic dynamics}

\subsubsection{Subshifts}

Let $\AA$ be a finite set that we call \emph{alphabet}.
Elements in $\AA$ are called \emph{letters} or \emph{symbols}.
The number of letters of $\AA$ is denoted by $|\AA|$.
The set of finite sequences or \emph{words} of length $\ell \in \NN$ with letters in $\AA$ is denoted by $\AA^\ell$ and the set of two-sided sequences $(x_n)_{n \in \ZZ}$ in $\AA$ is denoted by $\AA^\ZZ$.
A word $w = w_0 w_1 \ldots w_{\ell - 1} \in \AA^\ell$ can be seen as an element of the free monoid $\AA^\ast$ endowed with the operation of concatenation (whose neutral element is $\epsilon$, the empty word).
The integer $\ell$ is the \emph{length} of the word $w$ and is denoted by $|w| = \ell$; the length of the empty word is $0$.
A word $v$ is a \emph{power} of a word $u$ if $v = u^n$ for some $n \in \NN$.

For finite words $p$ and $s$ in $\AA^\ast$, we say that they are respectively a \emph{prefix} and a \emph{suffix} of the word $p s$.
For $x \in \AA^\ZZ$ and integers $N > n$ we define the word $x_{[n,N)} = x_n x_{n + 1} \ldots x_{N - 1}$.
For a nonempty word $w \in \AA^\ast$ and a point $x \in \AA^\ZZ$, we say that $w$ \emph{occurs} in $x$ if there exists $n \in \ZZ$ such that $x_n x_{n+1} \ldots x_{n + |w| - 1} = w$.
In this case, we say that the index $n$ is an \emph{occurrence} of $w$ in $x$.
We use the same notion for finite nonempty words $x$.
We say that a nonempty word $w = w_0 w_1 \ldots w_{\ell - 1} \in \AA^\ast$ \emph{starts} (resp. \emph{ends}) with a nonempty word $u \in \AA^\ast$ if $u = w_0 \ldots w_{i - 1}$ for some $i \le \ell$ (resp. $u = w_j \ldots w_{\ell - 1}$ for some $j \ge 0$).

The \emph{shift map} $S : \AA^\ZZ \to \AA^\ZZ$ is defined by $S ((x_n)_{n \in \ZZ}) = (x_{n+1})_{n \in \ZZ}$.
A \emph{subshift} is a topological dynamical system $(X, S)$ where $X$ is a closed and $S$-invariant subset of $\AA^\ZZ$.
Here, we consider the product topology on $\AA^\ZZ$.
Classically, one identifies $(X, S)$ with $X$, so one says that $X$ itself is a subshift.
When we say that a sequence $x$ in a subshift is \emph{aperiodic}, we implicitly mean that $x$ is aperiodic for the action of the shift.

Let $(X, S)$ be a subshift.
The \emph{language} of $(X, S)$ is the set $\LL(X)$ containing all words $w \in \AA^\ast$ such that $w = x_{[m, m + |w|)}$ for some $x = (x_n)_{n \in \ZZ} \in X$ and $m \in \ZZ$.
In this case, we also say that $w$ is a \emph{factor} (also called \emph{subword}) of $x$.
We denote by $\LL_\ell(X)$ the set of words of length $\ell$ in $\LL(X)$.
Given $x \in X$, the language $\LL(x)$ is the set of all words that occur in $x$.
As before, we define $\LL_\ell (x)$.
For two words $u,v \in \LL(X)$, the \emph{cylinder set} $[u.v]$ is the set $\{x \in X : x_{[-|u|,|v|)} = u v\}$.
When $u$ is the empty word we only write $[v]$, erasing the dot.
We remark that cylinder sets are clopen sets and they form a base for the topology of the subshift.

\subsubsection{Morphisms}\label{ss:morphisms}

Let $\AA$ and $\BB$ be finite alphabets and $\tau : \AA^\ast \to \BB^\ast$ be a morphism.
We say that $\tau$ is \emph{erasing} whenever there exists a letter $a \in \AA$ such that $\tau(a)$ is the empty word.
Otherwise we say it is \emph{nonerasing}.
When the morphism $\tau$ is nonerasing, it extends naturally to a map from $\AA^\ZZ$ to $\BB^\ZZ$ by concatenation (we apply $\tau$ to positive and negative coordinates separately and we concatenate the results at coordinate zero).
We continue to call this map $\tau$.
We observe that any map $\tau : \AA \to \BB^\ast$ can be naturally extended to a morphism (that we also denote by $\tau$) from $\AA^\ast$ to $\BB^\ast$ by concatenation.

The \emph{composition matrix} of a morphism $\tau : \AA^\ast \to \BB^\ast$ is given for each $a \in \AA$ and $b \in \BB$ by $M_\tau(b,a) = |\tau(a)|_b$, where $|\tau(a)|_b$ counts the number of occurrences of the letter $b$ in the word $\tau(a)$.
The morphism $\tau$ is said to be \emph{positive} if $M_\tau$ has positive entries and \emph{proper} if there exist $p, s \in \BB$ such that for all $a \in \AA$ the word $\tau(a)$ starts with $p$ and ends with $s$.

The minimum and maximal lengths of $\tau$ are, respectively, the numbers
\[
\langle \tau \rangle
= \min_{a \in \AA} |\tau(a)| \quad
\text{and} \quad
|\tau|
= \max_{a \in \AA} |\tau(a)|.
\]
We say that a morphism $\tau$ is of \emph{constant length} if $\langle \tau \rangle = |\tau|$.
Observe that if $\tau : \AA^\ast \to \BB^\ast$ and $\tau' : \BB^\ast \to \mathcal{C}^\ast$ are two constant length morphisms, then $\tau' \circ \tau$ is also of constant length and
\begin{equation}\label{eq:constant_length}
|\tau' \circ \tau|
= |\tau'| |\tau|.
\end{equation}

Following \cite{BSTY19}, a morphism $\tau : \AA^\ast \to \BB^\ast$ is \emph{left permutative} (resp. \emph{right permutative}) if the first (resp. last) letters of $\tau(a)$ and $\tau(b)$ are different, for all distinct letters $a,b \in \AA$.
Two morphisms $\tau, \widetilde{\tau} : \AA^\ast \to \BB^\ast$ are said to be \emph{rotationally conjugate} if there is a word $w \in \BB^\ast$ such that $\tau(a) w = w \widetilde{\tau}(a)$ for all $a \in \AA$ or $\widetilde{\tau}(a) w = w \tau(a)$ for all $a \in \AA$.

\subsubsection{$\SS$-adic subshifts}\label{ss:s-adic}

We recall the definition of \emph{$\SS$-adic subshifts} as stated in \cite{BSTY19}.
A \emph{directive sequence} $\btau = (\tau_n : \AA_{n+1}^\ast \to \AA_n^\ast)_{n \ge 0}$ is a sequence of nonerasing morphisms.
A slightly more general definition is given in \cite{DP22} including the case of erasing morphisms.
When all morphisms $\tau_n$ for $n \ge 0$ are proper we say that $\btau$ is \emph{proper}.
For $0 \le n \le N$, we denote by $\tau_{[n,N)}$ the morphism $\tau_n \circ \tau_{n+1} \circ \dots \circ \tau_{N-1}$, where $\tau_{[n,n)} : \AA_n^\ast \to \AA_n^\ast$ is the identity map for each $n \ge 0$.
We say $\btau$ is \emph{primitive} if for any $n \in \NN$ there exists $N > n$ such that $M_{\tau_{[n,N)}} $ has positive entries, \emph{i.e.}, for every $a \in \AA_N$ the word $\tau_{[n,N)}(a)$ contains all letters in $\AA_n$.

For $n \in \NN$, the \emph{language $\LL^{(n)}({\btau})$ of level $n$ associated 
with $\btau$} is defined by 
\[
\LL^{(n)}({\btau})
= \{w \in \AA_n^\ast : \text{$w$ occurs in $\tau_{[n,N)}(a)$
for some $a \in \AA_N$ and $N > n$}\}
\]
and let $X_{\btau}^{(n)}$ be the set of points $x \in \AA_n^\ZZ$ such that $\LL(x) \subseteq \LL^{(n)}({\btau})$.
This set clearly defines a subshift that we call the \emph{subshift generated by $\LL^{(n)}({\btau})$}.
We set $X_{\btau} = X_{\btau}^{(0)}$ and call $(X_{\btau},S)$ or $X_{\btau}$ the \emph{$\SS$-adic subshift} generated by the directive sequence $\btau$.

A \emph{contraction} of $\btau = (\tau_n : \AA_{n+1}^\ast \to \AA_n^\ast)_{n \ge 0}$ is a directive sequence of the form
\[
\widetilde{\btau}
= (\widetilde{\tau}_k = \tau_{[n_k, n_{k+1})} : \AA_{n_{k+1}}^\ast \to \AA_{n_k}^\ast)_{k \ge 0},
\]
where the sequence $(n_k)_{k \ge 0}$ is such that $n_0 = 0$ and $n_k < n_{k+1}$ for all $k \ge 0$.
Observe that any contraction of $\btau$ generates the same $\SS$-adic subshift $X_{\btau}$.

We say that a directive sequence $\btau = (\tau_n : \AA_{n+1}^\ast \to \AA_n^\ast)_{n \ge 0}$ is \emph{invertible} if the linear map $M_{\tau_n} : \RR^{\AA_n} \to \RR^{\AA_{n+1}}$ (acting on row vectors) is invertible for all $n \ge 0$.
Observe that this implies that the sequence $(|\AA_n|)_{n \ge 0}$ is constant.

If $\btau$ is primitive, then the subshift $(X_{\btau}, S)$ is minimal (see for instance \cite[Proposition 6.4.5]{DP22}).
The following proposition generalizes \cite[Lemma 3.3]{BCBD+21}.
The proof is similar and we include it here for the sake of completeness.

\begin{proposition}\label{p:aperiodic}
Let $\btau = (\tau_n : \AA_{n+1}^\ast \to \AA_n^\ast)_{n \ge 0}$ be a primitive and invertible directive sequence.
Then $(X_{\btau}, S)$ is minimal and aperiodic.
\end{proposition}

\begin{proof}
It is enough to show that $(X_{\btau}, S)$ is aperiodic.
By contradiction, define $p \in \NN$ to be the smallest possible period among all periodic points in $X_{\btau}$.

Let $y = \ldots u u . u u \ldots$ be a periodic point in $X_{\btau}$, where $|u| = p$.
Since $\btau$ is primitive, there exists $n \in \NN$ such that $\langle \tau_{[0,n)} \rangle \ge p$.
Without loss of generality, there exists $x \in \AA_n^\ZZ$ such that $y = \tau_{[0,n)}(x)$.
Furthermore, since $\btau$ is primitive we can assume that every letter of $\AA_n$ occurs in $x$.

If the word $\tau_{[0,n)}(x_0)$ is not a power of $u$, then there exists a nonempty prefix $v$ (resp. nonempty suffix $w$) of $u$ such that $u = v w$, $\tau_{[0,n)}(x_0)$ ends with $v$ and $\tau_{[0,n)}(x_1)$ starts with $w$.
The word $\tau_{[0,n)}(x_1)$ starts with $u$, so there exists a suffix $v'$ of $u$ such that $u = w v'$.
But since $y = \ldots u u . u u \ldots$, the word $v'$ is also a prefix of $u$ with $|v'| = |v|$, so $v = v'$.
Fine--Wilf's theorem then implies that $v$ and $w$ are powers of a same word, contradicting the definition of $p$.

This shows that $\tau_{[0,n)}(x_0) = u^{p_0}$ for some $p_0 \in \NN$ and, inductively, for each $m \in \ZZ$ there exists $p_m \in \NN$ such that $\tau_{[0,n)}(x_m) = u^{p_m}$.
In particular, for each $a \in \AA_n$ there exists $p_a \in \NN$ such that $\tau_{[0,n)}(a) = u^{p_a}$.
Therefore, the columns of $M_{\tau_{[0,n)}}$ are multiples of the column vector $(|u|_a)_{a \in \AA_0}$. This contradicts the fact that the linear map given by $M_{\tau_{[0,n)}}$ is invertible and finishes the proof.
\end{proof}

\subsubsection{Recognizability}\label{ss:recognizability}

Let $\tau : \AA^\ast \to \BB^\ast$ be a nonerasing morphism and $X \subseteq \AA^\ZZ$ be a subshift.
For $x \in X$ and $k \in \NN$ with $0 \le k < |\tau(x_0)|$, the \emph{cutting points} of the pair $(k, x)$ are defined as follows.
If $\ell \ge 0$, we define the $\ell$th cutting point of $(k, x)$ as 
\[
C_\tau^\ell (k, x)
= |\tau(x_{[0, \ell)})| - k.
\]
Similarly, if $\ell < 0$ the $\ell$th cutting point of $(k, x)$ is $C_\tau^\ell (k, x) = -|\tau(x_{[\ell, 0)})| - k$.
Define $\mathcal{C}_\tau^+(k, x) = \{C_\tau^\ell(k, x) : \ell > 0\}$.

If $y = S^k \tau(x)$ with $x \in X$ and $k \in \NN$, $0 \le k < |\tau(x_0)|$, we say that $(k, x)$ is a \emph{centered $\tau$-representation} of $y$.
The centered $\tau$-representation $(k, x)$ \emph{is in} $X$ if $x$ belongs to $X$.
The morphism $\tau$ is \emph{recognizable in $X$} (resp. \emph{recognizable in $X$ for aperiodic points}) if any point $y \in \BB^\ZZ$ (resp. any aperiodic point $y \in \BB^\ZZ$) has at most one centered $\tau$-representation in $X$.
If $\tau$ is recognizable in $\AA^\ZZ$ (for aperiodic points), we say that $\tau$ is \emph{fully recognizable} (for aperiodic points).

In the sequel, we use the following results \cite[Theorem 3.1, Lemma 3.5]{BSTY19}.

\begin{proposition}\label{p:right_perm}
Let $\tau : \AA^\ast \to \BB^\ast$ be a nonerasing morphism.
Assume that $\tau$ is (rotationally conjugate to) a left or right permutative morphism.
Then $\tau$ is fully recognizable for aperiodic points.
\end{proposition}

\begin{proposition}\label{p:composition}
Let $\sigma : \AA^\ast \to \BB^\ast$ and $\tau : \BB^\ast \to \mathcal{C}^\ast$ be two nonerasing morphisms, $X \subseteq \AA^\ZZ$ be a subshift and $Y = \bigcup_{k \in \ZZ} S^k \sigma(X)$.
If $\sigma$ is recognizable in $X$ for aperiodic points and $\tau$ is recognizable in $Y$ for aperiodic points, then $\tau \circ \sigma$ is recognizable in $X$ for aperiodic points.
\end{proposition}

We will also need the following straightforward lemma \cite[Proposition 1.4.30]{DP22}.

\begin{lemma}\label{l:induced}
Let $\tau : \AA^\ast \to \BB^\ast$ be a nonerasing morphism and $X \subseteq \AA^\ZZ$ be a minimal and aperiodic subshift.
Suppose that $\tau$ is recognizable in $X$ and let $Y = \bigcup_{k \in \ZZ} S^k \tau(X)$.
Then $(X, S)$ is topologically conjugate to the induced system $(\tau(X), S_{\tau(X)})$ of $(Y, S)$ on $\tau(X)$.
\end{lemma}

\subsubsection{Recognizability for sequences of morphisms}

Following \cite{BSTY19}, a directive sequence $\btau = (\tau_n : \AA_{n+1}^\ast \to \AA_n^\ast)_{n\ge 0}$ is said to be \emph{recognizable at level $n$} if the morphism $\tau_n$ is recognizable in $X_{\btau}^{(n+1)}$.
We say that the directive sequence $\btau$ is \emph{recognizable} if it is recognizable at level $n$ for each $n \ge 0$.

We have that $\btau$ is recognizable if and only if for all $0 \le n < N$ and any point $y \in X_{\btau}^{(n)}$ there is a unique couple $(k, x)$ with $x \in X_{\btau}^{(N)}$ and $0 \le k < |\tau_{[n,N)}(x_0)|$ such that $y = S^k \tau_{[n,N)}(x)$.
This is the content of \cite[Lemma 3.5, Lemma 4.2]{BSTY19}.
Indeed, $\btau$ is recognizable if and only if for all $n \ge 0$ and any point $y \in X_{\btau}$ there is a unique couple $(k, x)$ with $x \in X_{\btau}^{(n)}$ and $0 \le k < |\tau_{[0,n)}(x_0)|$ such that $y = S^k \tau_{[0,n)}(x)$.

\medskip

\Cref{l:induced} imply the following.

\begin{corollary}\label{c:induced_sadic}
Let $\btau = (\tau_n : \AA_{n+1}^\ast \to \AA_n^\ast)_{n \ge 0}$ be a recognizable directive sequence and let $\btau' = (\tau_{n+1} : \AA_{n+2}^\ast \to \AA_{n+1}^\ast)_{n \ge 0}$ be the shifted directive sequence.
Suppose that the subshift $(X_{\btau}, S)$ is minimal and aperiodic.
Then $(X_{\btau'}, S)$ is topologically conjugate to the induced system $(\tau_0(X_{\btau'}), S_{\tau_0(X_{\btau'})})$ of $(X_{\btau}, S)$ on $\tau_0(X_{\btau'})$.
\end{corollary}

\subsection{Kakutani--Rokhlin partitions}\label{s:kakutani_rokhlin}

Let $(X, T)$ be a minimal Cantor system.

\subsubsection{CKR partitions of minimal Cantor systems}

A \emph{clopen Kakutani--Rokhlin partition (CKR partition)} $\TT$ of $(X, T)$ is a partition of $X$ of the form 
\[
\TT
= \{T^k B(a) : a \in \AA(\TT),\ 
0 \le k < h(a)\},
\]
where $\AA(\TT)$ is a nonempty finite alphabet, the value $h(a)$ is a positive integer and $B(a)$ is a clopen set for all $a \in \AA(\TT)$.
Observe that
\[
\bigcup_{a \in \AA(\TT)} T^{h(a)} B(a)
= \bigcup_{a \in \AA(\TT)} B(a).
\]
The \emph{base} of $\TT$ is the set $B(\TT) = \bigcup_{a \in \AA(\TT)} B(a)$.
The set $\TT(a) = \bigcup_{0 \le k < h(a)} T^k B(a)$ is called the \emph{tower} indexed by $a \in \AA(\TT)$ of $\TT$ with \emph{base} $B(a)$ and \emph{height} $h(a)$.

Let
\[
\TT_n
= \{T^k B_n(a) : a \in \AA(\TT_n),\
0 \le k < h_n(a)\}, \quad
n \ge 0
\]
be a sequence of CKR partitions of $(X, T)$.
It is \emph{nested} if for any $n \ge 0$:

\begin{enumerate}
	\item[(KR1)] $B(\TT_{n+1}) \subseteq B(\TT_n)$;
	\item[(KR2)] $\TT_n \preceq \TT_{n+1}$, \emph{i.e.}, for every $A \in \TT_{n+1}$ there exists $B \in \TT_n$ such that $A \subseteq B$;
	\item[(KR3)] $\bigcap_{n \ge 0} B(\TT_n) = \{x\}$ for some point $x\in X$; and
	\item[(KR4)] the atoms of $\bigcup_{n \ge 0} \TT_n$ generate the topology of $X$.
\end{enumerate}
We remark that nested sequences always exist \cite[Theorem 4.2]{HPS92}.

For each $n \ge 0$, the \emph{incidence matrix} $M_n$ between the partitions $\TT_{n+1}$ and $\TT_n$ is given for each $a \in \AA(\TT_n)$ and $b \in \AA(\TT_{n+1})$ by
\begin{equation}
M_n(a,b)
= \# \{0 \le k < h_{n+1}(b) :
T^k B_{n+1}(b) \subseteq B_n(a)\}.
\end{equation}
For $n \ge 0$ let $h_n$ be the row vector called \emph{height vector} and defined by
\[
h_n 
= (h_n(a))_{a \in \AA(\TT_n)}.
\]
We define $P_{m,n} = M_m M_{m+1} \ldots M_{n-1}$ for $0 \le m < n$.
Observe that $P_{n,n+1} = M_n$.
By means of a simple induction argument, we have $h_n = h_m P_{m,n}$ and
\begin{equation}\label{eq:P_mn}
P_{m,n}(a,b)
= \# \{0 \le k < h_n(b) : T^k B_n(b) \subseteq B_m(a)\},
\end{equation}
$a \in \AA(\TT_m)$, $b \in \AA(\TT_n)$, $0 \le m < n$.

The \emph{topological rank} of $(X, T)$ is the value
\begin{equation}\label{eq:rank}
\rank (X, T)
= \inf_{\substack{\text{nested sequence $(\TT_n)_{n \ge 0}$} \\
\text{of CKR partitions of $(X, T)$}}} \liminf_{n \to +\infty} |\AA(\TT_n)|.
\end{equation}
Roughly speaking, the topological rank of $(X, T)$ is the smallest number of CKR towers needed to describe $(X, T)$.
The topological rank is invariant under topological conjugacy.
See \cites{DM08, BDM10} for more details.

\subsubsection{Invariant measures through CKR partitions}\label{ss:invariant_measures}

Let $(\TT_n)_{n \ge 0}$ be a nested sequence of CKR partitions.
Any $T$-invariant probability measure $\mu$ of $(X,T)$ is uniquely determined by the values it assigns to atoms of the partitions, hence to the bases $B_n(a)$, $a \in \AA(\TT_n)$ and $n \ge 0$.

For $n \ge 0$ let $\mu_n$ be the column vector called \emph{measure vector} and defined by
\[
\mu_n 
= (\mu_n(a))_{a \in \AA(\TT_n)}, \quad
\text{where} \quad
\mu_n(a)
= \mu(B_n(a)).
\]
Therefore, the measure $\mu$ is completely determined by the sequence of measure vectors $(\mu_n)_{n \ge 0}$.
Since $\mu$ is a probability measure we have 
\begin{equation}\label{eq:prob_measure}
\mu(\TT_n(a))
= h_n(a) \mu_n(a) \quad
\text{and} \quad
\sum_{a \in \AA(\TT_n)} \mu(\TT_n(a)) = 1.
\end{equation}
Additionally, by \eqref{eq:P_mn} we have 
\begin{equation}\label{eq:inv_measure}
\mu_m = P_{m,n} \mu_n, \quad
0 \le m < n.
\end{equation}

\subsubsection{CKR partitions of \SSS-adic subshifts}\label{ss:rep_sadic}

Let $\btau = (\tau_n : \AA_{n+1}^\ast \to \AA_n^\ast)_{n \ge 0}$ be a primitive, proper and recognizable directive sequence which generates the $\SS$-adic subshift $(X_{\btau}, S)$.
Define the sequence $(\TT_n)_{n \ge 0}$ as follows:
\begin{equation}\label{eq:rep_sadic}
\TT_n
= \{S^k \tau_{[0,n)}([a]) :
a \in \AA_n,\
0 \le k < |\tau_{[0,n)}(a)|\}, \quad
n \ge 0.
\end{equation}

The following result proved in \cite[Proposition 2.2]{DL12} shows that $(\TT_n)_{n \ge 0}$ defines a nested sequence of CKR partitions.
We include a proof for the sake of completeness.

\begin{proposition}\label{p:nested}
The sequence $(\TT_n)_{n \ge 0}$ is a nested sequence of CKR partitions of $(X_{\btau}, S)$.
Moreover, for each $n \ge 0$ the incidence matrix $M_n$ between the partitions $\TT_{n+1}$ and $\TT_n$ coincides with the composition matrix $M_{\tau_n}$ of the morphism $\tau_n$:
\[
M_n
= M_{\tau_n}.
\]
\end{proposition}

\begin{proof}
Since $\btau$ is recognizable, $\TT_n$ is a CKR partition of $X_{\btau}$ for each $n \ge 0$.
Observe that the tower $\TT_n(a)$ has base $B_n(a) = \tau_{[0,n)}([a])$ for $a \in \AA_n$.
Clearly we have $B(\TT_{n+1}) \subseteq B(\TT_n)$ for $n \ge 0$.

\begin{claim}
$\TT_n \preceq \TT_{n+1}$.
\end{claim}

Indeed, let $S^k \tau_{[0, n+1)}([a])$ be an atom of $\TT_{n+1}$, $a \in \AA_{n+1}$, $0 \le k < |\tau_{[0, n+1)}(a)|$.
Let $\tau_n(a) = b_0 b_1 \ldots b_{i-1}$ with $b_j \in \AA_n$, $0 \le j < i$. Then, there exists $j \in [0, i-1)$ satisfying
\[
|\tau_{[0,n)}(b_0 b_1 \ldots b_j)|
\le k
< |\tau_{[0,n)}(b_0 b_1 \ldots b_{j+1})|.
\]
We deduce that if $k' = |\tau_{[0,n)}(b_0 b_1 \ldots b_j)|$, then $S^k \tau_{[0,n+1)}([a]) \subseteq S^{k - k'} \tau_{[0,n)}([b_{j+1}])$ with $0 \le k - k' < |\tau_{[0,n)}(b_{j+1})|$.
This proves the claim.

\begin{claim}
The atoms of $\bigcup_{n \ge 0} \TT_n$ generate the topology of $X_{\btau}$.
\end{claim}

Indeed, let $n \ge 1$, $a \in \AA_n$, $0 \le k < |\tau_{[0,n)}(a)|$ and $\ell$ be a nonnegative integer.
As $\tau_n$ is proper, there exist two letters $p_n$ and $s_n$ in $\AA_n$ such that $\tau_n(a)$ starts with $p_n$ and ends with $s_n$ for all $a \in \AA_{n+1}$ and $n \ge 0$.
Since $\btau$ is primitive, there exists $N \in \NN$ such that if $n \ge N$ then $\langle \tau_{[0,n-1)} \rangle \ge \ell$.
Let $x', y' \in \tau_{[0,n)}([a])$, $u_n = \tau_{[0,n-1)}(s_{n-1})$ and $v_n = \tau_{[0,n)}(a) \tau_{[0,n-1)}(p_{n-1})$.
We have
\[
x'_{[-|u_n|, |v_n|)}
= y'_{[-|u_n|, |v_n|)}
= u_n v_n,
\]
so that $x'_{[-\ell, \ell + k]} = y'_{[-\ell, \ell + k]}$.
If $x,y$ belong to $S^k \tau_{[0,n)}([a])$, $n \ge N$, then $x_{[-\ell - k, \ell]} = y_{[-\ell - k, \ell]}$, and in particular $x_{[-\ell, \ell]} = y_{[-\ell, \ell]}$.
Therefore $\diam(S^k \tau_{[0,n)}([a])) \to 0$ as $n \to +\infty$.
This proves the claim.
Since the bases $(B(\TT_n))_{n \ge 0}$ are nested, they converge to some point.
This finishes the proof of the first statement.


The second statement follows easily from the recognizability of $\btau$.
\end{proof}

We remark that the height vectors $(h_n)_{n \ge 0}$ of $(\TT_n)_{n \ge 0}$ defined by \eqref{eq:rep_sadic} satisfy
\begin{equation}\label{eq:heights}
h_n(a) = |\tau_{[0, n)}(a)|, \quad
a \in \AA_n, \quad
n \ge 0.
\end{equation}

\subsection{Dimension groups}\label{ss:dimension_groups}

In this section we recall the basic on dimension groups and state the main results that we will use throughout this article.
We refer to \cites{GPS95, DP22} for more complete references.

\subsubsection{Direct limits}

Let $(G_n)_{n \ge 0}$ be a sequence of abelian groups and let $i_{n+1, n} : G_n \to G_{n+1}$ for each $n \ge 0$ be a morphism.
Define the subgroups $\Delta$ and $\Delta^0$ of the direct product $\prod_{n \ge 0} G_n$ by
\[
\Delta = 
\{(g_n)_{n \ge 0} \in \textstyle\prod_{n \ge 0} G_n :
g_{n+1}
= i_{n+1, n}(g_n)\
\text{for every large enough $n$}\}
\]
and
\[\Delta^0 = 
\{(g_n)_{n \ge 0} \in \textstyle\prod_{n \ge 0} G_n :
g_n = 0\
\text{for every large enough $n$}\}.
\]
Let $G = \Delta / \Delta^0$ be the quotient group and $\pi : \Delta \to G$ be the natural projection.
The group $G$ is called the \emph{direct limit} of $(G_n)_{n \ge 0}$ and we write $G = \varinjlim G_n$.
If $g \in G_n$, then all sequences $(g_k)_{k \ge 0}$ such that $g_n = g$ and $g_{k+1} = i_{k+1, k}(g_k)$ for all $k \ge n$ belong to $\Delta$ and have the same projection in $G$, denoted by $i_n(g)$.
This defines a group morphism $i_n : G_n \to G$, which we call the \emph{natural morphism} from $G_n$ to $G$.
For $0 \le m < n$ define
\[
i_{n, m}
= i_{n-1, n} \circ i_{n, n+1} \circ \ldots \circ i_{m+1, m}.
\]
We have $i_m = i_n \circ i_{n, m}$ and $G = \bigcup_{n \ge 0} \ima i_n$.

We can also define direct limits of vector spaces.
Let $\mathbb{K}$ be a field.
For each $n \ge 0$, let $V_n$ be a vector space over $\mathbb{K}$ and $i_{n+1, n} : V_n \to V_{n+1}$ be a linear map.
The \emph{direct limit} $V = \varinjlim V_n$ is the vector space over $\mathbb{K}$, where the group structure on $V$ is the one given by the direct limit of the abelian groups $V_n$ and the scalar multiplication is given by pointwise scalar multiplication on each coordinate.

\subsubsection{Orbit equivalence}

Two minimal Cantor systems $(X, T)$ and $(X', T')$ are \emph{orbit equivalent} if there exists a homeomorphism $\Phi : X \to X'$ which sends orbits onto orbits, \emph{i.e.},
\[
\Phi(\{T^n x : n \in \ZZ\})
= \{(T')^n \circ \Phi (x) : n \in \ZZ\}, \quad
x \in X.
\]
This implies that there exist two maps $\alpha : X \to \ZZ$ and $\beta : X' \to \ZZ$, uniquely defined by aperiodicity, such  that
\[
\Phi \circ T (x)
= (T')^{\alpha(x)} \circ \Phi(x) \quad
\text{and} \quad
\Phi \circ T^{\beta(x)} (x)
= T' \circ \Phi(x), \quad
x \in X.
\]
The minimal Cantor systems $(X, T)$ and $(X', T')$ are \emph{strongly orbit equivalent} if $\alpha$ and $\beta$ both have at most one point of discontinuity.

\subsubsection{Dimension groups of minimal Cantor systems}

Denote by $C(X, \ZZ)$ (resp. $C(X, \NN)$) the group (resp. monoid) of continuous functions from $X$ to $\ZZ$ (resp. $\NN$) with the addition operation.
Consider the map $\partial : C(X, \ZZ) \to C(X, \ZZ)$ defined by $\partial f = f \circ T - f$.

A map $f$ is called a \emph{coboundary} if there exists $g \in C(X, \ZZ)$ such that $f = \partial g$.
Two maps $f, f' \in C(X, \ZZ)$ are said to be \emph{cohomologous} if $f - f'$ is a coboundary.

Define the quotient group $H(X, T) = C(X, \ZZ) / \partial C(X, \ZZ)$.
Let $[f]$ be the class of $f \in C(X, \ZZ)$ in $H(X,T)$ and $\pi : C(X, \ZZ) \to H(X,T)$ be the projection map.
Define $H^+(X, T) = \pi(C(X, \NN))$ and denote by $\bm{1}_X$ the constant one valued function.

Consider the triple
\[
K^0(X, T)
= (H(X, T), H^+(X, T), [\bm{1}_X]).
\]
It is an \emph{ordered group} with \emph{order unit} $[\bm{1}_X]$.
As $(X, T)$ is minimal, it is a \emph{dimension group}.
See \cite{GPS95, DP22} for the definitions and more details.
We call it \emph{the dimension group of $(X, T)$}.

It is classical to observe that if $(X, T)$ is topologically conjugate to $(X', T')$, then the ordered groups with order units $K^0(X, T)$ and $K^0(X', T')$ are \emph{unital order isomorphic}, \emph{i.e.}, there exists a group morphism $\delta : H(X, T) \to H(X', T')$ such that  $\delta (H^+(X, T))= H^+(X', T')$ and $\delta ([\bm{1}_X]) = [\bm{1}_{X'}]$.

Denote by $\mathcal{M}(X, T)$ the set of invariant probability measures of $(X, T)$.
We define the set of \emph{infinitesimals} of $H(X, T)$ as
\[
\Inf H(X, T)
= \Big\{[f] \in H(X, T) : \int f d\mu
= 0\ \text{for all}\ \mu \in \mathcal{M}(X, T)\Big\}.
\]
We have that $H(X, T) / \Inf H(X, T)$ with the induced order is also a dimension group.
We denote it $K^0(X, T) / \Inf K^0(X, T)$.

The dimension groups $K^0(X, T)$ and $K^0(X, T) / \Inf K^0(X, T)$ characterize \emph{strong orbit equivalence} and \emph{orbit equivalence}, respectively \cite{GPS95}.

Another description of the dimension group $K^0(X, T)$ is as follows.
Let $(\TT_n)_{n \ge 0}$ be a nested sequence of CKR partitions of $(X, T)$ as defined in \Cref{s:kakutani_rokhlin}.
Let $(\AA(\TT_n))_{n \ge 0}$, $(h_n)_{n \ge 0}$ and $(M_n)_{n \ge 0}$ be the associated sequences of alphabets, height vectors and incidence matrices, respectively.

For $n \ge 0$ we consider $\ZZ^{\AA(\TT_n)}$ as an ordered group of row vectors with the usual order.
Define the sequence of ordered groups with order units
\[
\GG_n
= (\ZZ^{\AA(\TT_n)}, \ZZ_+^{\AA(\TT_n)}, h_n), \quad
n \ge 0.
\]
Let $\GG = \varinjlim \GG_n$ be the direct limit of the groups $\GG_n$ with respect to the morphisms $M_n : \ZZ^{\AA(\TT_n)} \to \ZZ^{\AA(\TT_{n+1})}$ given by the incidence matrix $M_n$ (acting on row vectors).
Let $\GG^+$ be the projection in $\GG$ of the set of points $(x_n)_{n \ge 0} \in \textstyle\prod_{n \ge 0} \ZZ^{\AA(\TT_n)}$ for which there exists $N \in \NN$ such that $x_N \in \ZZ_+^{\AA(\TT_N)}$ and $x_{k+1} = x_k M_k$, $k \ge N$.
Denote by $u$ the projection in $\GG$ of the sequence $(h_n)_{n \ge 0}$.

The tuple $\mathcal{K} = (\GG, \GG^+, u)$ is a dimension group.
The introduction of this dimension group is motivated by the following proposition \cite[Theorem 5.3.6]{DP22}.

\begin{proposition}\label{p:dimension_group_CKR}
Let $(X, T)$ be a minimal Cantor system and $\mathcal{K}$ be the dimension group associated to a nested sequence of CKR partitions of $(X, T)$.
Then, the dimension group $K^0(X, T)$ is unital order isomorphic to $\mathcal{K}$.
\end{proposition}

\subsubsection{Dimension groups of $\SS$-adic subshifts}\label{ss:dg_sadic}

Let $\btau = (\tau_n : \AA_{n+1}^\ast \to \AA_n^\ast)_{n \ge 0}$ be a primitive, proper and recognizable directive sequence.
Suppose that $(X_{\btau}, S)$ is aperiodic.
Let $(\TT_n)_{n \ge 0}$ be the sequence of CKR partitions given in \eqref{eq:rep_sadic} and $\mathcal{K}$ be the dimension group associated to it.
Recall that, by \Cref{p:nested}, the incidence matrix $M_n$ between the partitions $\TT_{n+1}$ and $\TT_n$ coincide with the composition matrix $M_{\tau_n}$ of the morphism $\tau_n$.

We call $\mathcal{K}$ \emph{the dimension group of $\btau$}.
By \Cref{p:dimension_group_CKR}, the dimension group of $(X_{\btau}, S)$ is unital order isomorphic to $\mathcal{K}$.

In the case where all the linear maps $M_{\tau_n}$, $n \ge 1$ are invertible, it is easy to check from the definition that the dimension group $K^0 (X_{\btau}, S)$ is unital order isomorphic to $(\GG, \GG^+, u)$, where
\begin{align*}
\GG &=
\{x \in \RR^{\AA_1} :
x M_{\tau_1} M_{\tau_2} \ldots M_{\tau_n} \in \ZZ^{\AA_{n+1}}\
\text{for large enough $n$}\};\\
\GG^+ &=
\{x \in \RR^{\AA_1} :
x M_{\tau_1} M_{\tau_2} \ldots M_{\tau_n} \in \ZZ_+^{\AA_{n+1}}\
\text{for large enough $n$}\};
\end{align*}
and $u = (|\tau_0(a)|)_{a \in \AA_1} \in \RR^{\AA_1}$.

\section{\SSS-adic representation of minimal Ferenczi subshifts}\label{s:ferenczi_subshifts}

\subsection{Ferenczi subshifts}\label{ss:ferenczi_subshifts}

Following \cites{Fer96, Fer97}, we consider sequences of nonnegative integers $(q_n)_{n \ge 0}$ and $(a_{n,i} : n \ge 0,\ 0 \le i < q_n)$, which we call \emph{cutting} and \emph{spacers} parameters, respectively.
These parameters define a sequence of \emph{generating words} $\WW = (w_n)_{n \ge 0}$ over the alphabet $\{0, 1\}$ inductively by
\begin{equation}\label{eq:generating_words}
w_0
= 0 \quad
\text{and} \quad
w_{n+1}
= w_n 1^{a_{n,0}} w_n 1^{a_{n,1}} \ldots w_n 1^{a_{n,q_n - 1}} w_n, \quad
n \ge 0.
\end{equation}
Observe that
\begin{equation}\label{eq:lengths}
|w_{n+1}|
= (q_n + 1) |w_n| + \sum_{i=0}^{q_n - 1} a_{n,i}, \quad
n \ge 0.
\end{equation}

The sequence $\WW$ allows to construct the subspace of $\{0, 1\}^\ZZ$ given by
\[
X_\WW
= \{x \in \{0, 1\}^\ZZ :
\text{every factor of $x$ is a factor of $w_n$ for some $n \ge 0$}\}
\]
and a one-sided sequence $x \in \{0,1\}^\NN$ by
\begin{equation}\label{eq:infinite_sequence}
x_{[0, |w_n|)}
= w_n, \quad
n \ge 0.
\end{equation}
We define
\begin{equation}\label{eq:Qmn}
Q_{m,n} = \prod\limits_{j = m}^{n-1} (q_j + 1), \quad
0 \le m < n.
\end{equation}
A \emph{contraction} of $\WW$ is a sequence of generating words of the form $\widetilde{\WW} = (w_{n_k})_{k \ge 0}$, where the sequence $(n_k)_{k \ge 0}$ is such that $n_0 = 0$ and $n_k < n_{k+1}$ for all $k \ge 0$.
Observe that if $\widetilde{\WW}$ is a contraction of $\WW$, then the generating words of $\widetilde{\WW}$ satisfy a relation of type \eqref{eq:generating_words} with new parameters $(\widetilde{q}_k : k \ge 0)$ such that 
\begin{equation}\label{eq:cutting}
\widetilde{q}_k + 1
= Q_{n_k, n_{k+1}}, \quad
k \ge 0.
\end{equation}
Moreover, it is easy to check that $X_{\widetilde{\WW}} = X_\WW$.

The pair $(X_\WW, S)$ is a subshift, which we call the \emph{Ferenczi subshift} associated to $\WW$.
It is minimal if the sequence $(a_{n,i} : n \ge 0,\ 0 \le i < q_n)$ is bounded.
If such a sequence is otherwise unbounded, then the two-sided sequence $1^\infty$ given by $1_n^\infty = 1$ for all $n \in \ZZ$ belongs to $X_\WW$ and $X_\WW$ contains at least two points, in particular the subshift $(X_\WW, S)$ is not minimal.
Moreover, in the minimal case, $X_\WW$ is finite if and only if the sequence $x$ given by \eqref{eq:infinite_sequence} is periodic.
See \cite[Section 2]{GH16a}.

\medskip

In the next section we prove that minimal Ferenczi subshifts are $\SS$-adic subshifts.
This is summarized in \Cref{p:tilde_tau}.

\subsection{Minimal Ferenczi subshifts are \SSS-adic}\label{ss:directive_sequences}

From now on, we will assume that $(X_\WW, S)$ is a minimal and aperiodic Ferenczi subshift.
Let $\{a_1, a_2, \ldots, a_\ell\}$ be the set of values of the sequence $(a_{n,i} : n \ge 0,\ 0 \le i < q_n)$ with $a_1 < a_2 < \ldots < a_\ell$.

We begin by constructing a sequence of alphabets $(\AA_n)_{n \ge 0}$ as follows.
Define $\AA_0 = \{0, 1\}$ and for $n \ge 1$ we set
\[
\AA_n
= \{a : a = a_{N,i}\
\text{for some $N \ge n-1$ and $0 \le i < q_N$}\}.
\]
In particular, we have $\AA_1 = \{a_1, a_2, \ldots, a_\ell\}$ and $\AA_n $ is included in $\AA_m$ if $1 \le m \le n$.
Consequently, there exists $n_0 \in \NN$ such that $\AA_n = \AA_{n_0}$ for all $n \ge n_0$.
We define
\begin{equation}\label{eq:A_WW}
\AA_\WW = \AA_{n_0} \quad
\text{and} \quad
d_\WW = |\AA_\WW|.
\end{equation}
It is easy to see that $\AA_\WW$ is well-defined and that if $\WW'$ is a contraction of $\WW$, then $\AA_\WW = \AA_{\WW'}$.
Moreover, since $(X_\WW, S)$ is aperiodic, we have $d_\WW \ge 2$.
Indeed, suppose that $\AA_\WW = \{a\}$ for some $a$. 
Then, one has that $w_n = w_{n_0} 1^a w_{n_0} 1^a \ldots  w_{n_0} 1^a w_{n_0}$ for $n \ge n_0$, contradicting the aperiodicity.

Define the morphism $\tau_0 : \AA_1^\ast \to \AA_0^\ast$ by $\tau_0(a) = 0 1^a$ for $a \in \AA_1$ and the morphism $\widetilde{\tau}_n : \AA_{n+1}^\ast \to \AA_n^\ast$ by
\begin{equation}\label{eq:tau_tilde}
\widetilde{\tau}_n(a)
= a_{n-1,0} a_{n-1,1} \ldots a_{n-1,q_{n-1}-1} a,
\quad a \in \AA_{n+1},
\quad n \ge 1.
\end{equation}
Each morphism $\widetilde{\tau}_n$ for $n \ge 1$ is well-defined, of constant length and right permutative.
Indeed, the images of letters under $\widetilde{\tau}_n$ differ only at the last letter.

We define the directive sequence $\widetilde{\btau}_\WW = (\widetilde{\tau}_n : \AA_{n+1}^\ast \to \AA_n^\ast)_{n \ge 0}$, where $\widetilde{\tau}_0 = \tau_0$.

\begin{lemma}\label{l:natural_subs}
We have $\widetilde{\tau}_{[0,n+1)}(a) = w_n 1^a$ for all $n \ge 0$ and $a \in \AA_{n+1}$.
\end{lemma}

\begin{proof}
By induction, the property holds if $n = 0$ since $w_0 = 0$.
Now if the property holds for $n \ge 0$, then for $a \in \AA_{n+2}$ we have
\[
\widetilde{\tau}_{[0,n+2)}(a)
= \widetilde{\tau}_{[0,n+1)}(a_{n,0} a_{n,1}
\ldots a_{n,q_n-1} a)
= w_n 1^{a_{n,0}} w_n 1^{a_{n,1}} \ldots w_n 1^{a_{n,q_n-1}} w_n 1^a,
\]
which is precisely $w_{n+1} 1^a$ by \eqref{eq:generating_words}, proving the property by induction.
\end{proof}

\begin{lemma}\label{l:primitive}
The directive sequence $\widetilde{\btau}_\WW$ is primitive.
Moreover, for all $n \ge 0$ and $a \in \AA_{n+1}$, there exists $N > n$ such that $w_n 1^a$ is a factor of the word $w_N$.
\end{lemma}

\begin{proof}
For the first assumption let $n \ge 0$.
One has to find $N > n$ such that $M_{\widetilde{\tau}_{[n, N)}}$ has positive entries.
If $n = 0$ this is given by \Cref{l:natural_subs}.
Suppose $n \ge 1$.
If $a$ belongs $\AA_n$ then, by definition, there exists $N \ge n-1$ such that $a_{N,i} = a$ for some $0 \le i < q_N$.
This implies that $a$ has an occurrence in $\widetilde{\tau}_{N+1}(b)$ for all $b \in \AA_{N+2}$, and hence, by \eqref{eq:tau_tilde}, it also has an occurrence in $\widetilde{\tau}_{[n,N+2)}(b)$.
This proves the first claim.

For the second one, let $n \ge 0$ and $a \in \AA_{n+1}$.
Since $\widetilde{\btau}_\WW$ is primitive, there exist $N > n$ and $b \in \AA_{N+1}$ such that $\widetilde{\tau}_{[n+1, N+1)}(b) = u a v b$ for some words $u, v$.
Hence, by \Cref{l:natural_subs} we obtain
\[
w_N 1^b
= \widetilde{\tau}_{[0,N+1)}(b)
= \widetilde{\tau}_{[0,n+1)}(u a v b)
= u' w_n 1^a v' w_n 1^b,
\]
for some words $u', v'$, and thus $w_n 1^a$ is a factor of the word $w_N$.
\end{proof}

Now we prove that the directive sequence $\widetilde{\btau}_\WW$ generates the subshift $X_\WW$.

\begin{proposition}\label{p:tilde_tau}
We have $X_\WW = X_{\widetilde{\btau}_\WW}$.
\end{proposition}

\begin{proof}
If $x$ belongs to $X_\WW$, then every factor of $x$ is a factor of some generating word $w_n$ for some $n \ge 0$, and hence also a factor of $w_n 1^a = \widetilde{\tau}_{[0,n+1)}(a)$ for some $a \in \AA_{n+1}$ by \Cref{l:natural_subs}.
Thus $x$ belongs to $X_{\widetilde{\btau}_\WW}$ and $X_\WW$ is included in $X_{\widetilde{\btau}_\WW}$.

If now $x$ belongs to $X_{\widetilde{\btau}_\WW}$, then every factor of $x$ is a factor of $\widetilde{\tau}_{[0,n)}(a) = w_{n-1} 1^a$ for some $n \ge 1$ and $a \in \AA_n$, thus also a factor of $w_N$ for some $N \ge n$ by \Cref{l:primitive}.
We conclude that $x$ belongs to $X_\WW$ and $X_{\widetilde{\btau}_\WW}$ is included in $X_\WW$.
\end{proof}

\subsection{Recognizable directive sequences for minimal Ferenczi subshifts}\label{ss:recognizable_sequence}

In this section, by a slight modification of the directive sequence $\widetilde{\btau}_\WW$, we describe a primitive, proper and recognizable directive sequence $\btau_\WW$ generating the minimal Ferenczi subshift $(X_\WW, S)$.
This is summarized in \Cref{t:recognizable_sequence}.

We say that the sequence of generating words $\WW$ is \emph{standard} if the sequence $(q_n)_{n \ge 0}$ given by \eqref{eq:generating_words} satisfies $q_n \ge 2$ for each $n \ge 0$.
Observe that we can assume without loss of generality that each sequence $\WW$ is standard.
Indeed, this follows directly from Equation \eqref{eq:cutting}.
From now on assume that $\WW$ is standard.

In order to apply \Cref{p:nested} and obtain sequences of CKR partitions for the subshift $(X_\WW, S)$, we need each morphism $\widetilde{\tau}_n$ for $n \ge 1$ to be proper, which is not the case.
We define a morphism $\tau_n : \AA_{n+1}^\ast \to \AA_n^\ast$ which is proper and rotationally conjugate (as defined in \Cref{ss:morphisms}) to $\widetilde{\tau}_n$ by
\begin{equation}\label{eq:tau}
\tau_n(a)
= a_{n-1,1} a_{n-1,2} \ldots a_{n-1,q_{n-1}-1} a a_{n-1,0},
\quad a \in \AA_{n+1},
\quad n \ge 1.
\end{equation}
Since $\WW$ is standard, this is a well-defined proper morphism of constant length which is rotationally conjugate to $\widetilde{\tau}_n$:
\begin{equation}\label{eq:rotationally}
a_{n-1,0} \tau_n(a)
= \widetilde{\tau}_n(a) a_{n-1,0}, \quad
a \in \AA_{n+1}, \quad
n \ge 1.
\end{equation}

We define the directive sequence $\btau_\WW = (\tau_n : \AA_{n+1}^\ast \to \AA_n^\ast)_{n \ge 0}$.

\begin{lemma}\label{l:rotations_sigma}
For $a \in \AA_{m+1}$ and $m \ge 1$ we have
\begin{align*}
&a_{0,0} \tau_{[1,2)}(a_{1,0}) \ldots \tau_{[1,m)}(a_{m-1,0}) \tau_{[1,m+1)}(a)\\
=& \widetilde{\tau}_{[1,m+1)}(a) \widetilde{\tau}_{[1,m)}(a_{m-1,0}) \ldots \widetilde{\tau}_{[1,2)}(a_{1,0}) a_{0,0}.
\end{align*}
\end{lemma}

\begin{proof}
We begin by proving the following.
\begin{claim}
For $1 \le n \le N$ and $a \in \AA_{N+1}$, we have
\[
\tau_{[n,N)}(a_{N-1,0}) \tau_{[n,N+1)}(a)
= \tau_{[n,N)}(\widetilde{\tau}_N(a)) \tau_{[n,N)}(a_{N-1,0}).
\]
\end{claim}
Indeed, by means of a simple computation
\begin{align*}
&\tau_{[n,N)}(a_{N-1,0}) \tau_{[n,N+1)}(a)\\
=& \tau_{[n,N)}(a_{N-1,0}) \tau_{[n,N)}(\tau_N(a))\\
=& \tau_{[n,N)}(a_{N-1,0}) \tau_{[n,N)}(a_{N-1,1} \ldots
a_{N-1,q_{N-1} - 1} a a_{N-1,0})\\
=& \tau_{[n,N)}(a_{N-1,0} \ldots a_{N-1,q_{N-1} - 1} a)
\tau_{[n,N)}(a_{N-1,0})\\
=& \tau_{[n,N)}(\widetilde{\tau}_N(a))
\tau_{[n,N)}(a_{N-1,0}),
\end{align*}
proving the claim.
This implies that for $n \in [1, N]$ and $w \in \AA_{N+1}^\ast$, then
\begin{equation}\label{eq:claim}
\tau_{[n,N)}(a_{N-1,0}) \tau_{[n,N+1)}(w)
= \tau_{[n,N)}(\widetilde{\tau}_N(w))
\tau_{[n,N)}(a_{N-1,0}).
\end{equation}

By induction, the statement in the lemma is true if $m = 1$ (see \eqref{eq:rotationally}).
Assume that the statement holds for $m \ge 1$.
By the claim, for $a \in \AA_{m+2}$ we obtain
\begin{align*}
&a_{0,0} \tau_{[1,2)}(a_{1,0}) \ldots \tau_{[1,m)}(a_{m-1,0}) \tau_{[1,m+1)}(a_{m,0}) \tau_{[1,m+2)}(a)\\
=& a_{0,0} \tau_{[1,2)}(a_{1,0}) \ldots \tau_{[1,m)}(a_{m-1,0}) \tau_{[1,m+1)}(\widetilde{\tau}_{m+1}(a)) \tau_{[1,m+1)}(a_{m,0}).
\end{align*}
By using \eqref{eq:claim} with $w = \widetilde{\tau}_{[k, m+2)}(a)$ for $k = m+1, m, \ldots, 2$ and the induction hypothesis, the last term is equal to
\begin{align*}
&\widetilde{\tau}_{[1,m+2)}(a) a_{0,0} \tau_{[1,2)}(a_{1,0}) \ldots \tau_{[1,m+1)}(a_{m,0})\\
=& \widetilde{\tau}_{[1,m+2)}(a) \widetilde{\tau}_{[1,m+1)}(a_{m,0}) \ldots \widetilde{\tau}_{[1,2)}(a_{1,0}) a_{0,0},
\end{align*}
finishing the proof by induction.
\end{proof}

We now prove that the sequences $\btau_\WW$ and $\widetilde{\btau}_\WW$ generate the same subshift.

\begin{proposition}\label{p:s_adic_ferenczi}
We have $X_\WW = X_{\btau_\WW}$.
\end{proposition}

\begin{proof}
By \Cref{p:tilde_tau}, it is enough to show that $X_{\btau_\WW} = X_{\widetilde{\btau}_\WW}$.
\begin{claim}
The word $a_{0,0} \tau_{[1,2)}(a_{1,0}) \ldots \tau_{[1,n)}(a_{n-1,0})$ is a suffix of $\tau_{[1,n+1)}(a)$ for all $a \in \AA_{n+1}$ and $n \ge 1$.
\end{claim}
Indeed, this is true for $n = 1$.
Assume that the claim holds for $n \ge 1$.
If $a$ belongs to $\AA_{n+2}$, then the word $a_{0,0} \tau_{[1,2)}(a_{1,0}) \ldots \tau_{[1,n)}(a_{n-1,0}) \tau_{[1,n+1)}(a_{n,0})$ is a suffix of the word
\begin{align*}
&\tau_{[1,n+1)}(a_{n,1}) \ldots
\tau_{[1,n+1)}(a_{n,q_n - 1}) \tau_{[1,n+1)}(a) \tau_{[1,n+1)}(a_{n,0})\\
=& \tau_{[1,n+1)}(a_{n,1} \ldots a_{n,q_n - 1} a a_{n,0}) = \tau_{[1,n+2)}(a),
\end{align*}
proving the claim by induction.

Let $x \in X_{\widetilde{\btau}_\WW}$ and $w$ be a factor of $x$.
Then $w$ is a factor of $\tau_0 \circ \widetilde{\tau}_{[1,n+1)}(a)$ for some $n \ge 1$ and $a \in \AA_{n+1}$.
By \Cref{l:rotations_sigma}, we deduce that $w$ is a factor of the word
\[
\tau_0(a_{0,0} \tau_{[1,2)}(a_{1,0}) \ldots \tau_{[1,n)}(a_{n-1,0})) \tau_0(\tau_{[1,n+1)}(a)).
\]
By using the previous claim with $a = a_{n,q_n - 1}$, the word $w$ is a factor of 
\[\tau_0(\tau_{[1,n+1)}(a_{n,q_n - 1})) \tau_0(\tau_{[1,n+1)}(a)),\]
and by \eqref{eq:tau} also a factor of $\tau_0 \circ \tau_{[1,n+2)}(a)$.
Thus $x$ belongs to $X_{\btau_\WW}$ and $X_{\widetilde{\btau}_\WW}$ is included in $X_{\btau_\WW}$.
Proving a similar claim reversing the roles of $X_{\btau_\WW}$ and $X_{\widetilde{\btau}_\WW}$, we obtain that $X_{\btau_\WW}$ is included in $X_{\widetilde{\btau}_\WW}$.
\end{proof}

We observe that the directive sequence $\btau_\WW$ is primitive.
Indeed, this follows directly from \Cref{l:primitive} since $\tau_n$ is rotationally conjugate to $\widetilde{\tau}_n$ for each $n \ge 1$.


\begin{lemma}\label{l:recognizability}
The directive sequences $\btau_\WW$ and $\widetilde{\btau}_\WW$ are recognizable.
\end{lemma}

\begin{proof}
Let $y \in X_\WW$ be any aperiodic point.
We prove the uniqueness of a couple $(k, x)$ with $x \in \AA_1^\ZZ$, $0 \le k < |\tau_0(x_0)|$ such that $y = S^k \tau_0(x)$.
Indeed, $y$ can be decomposed uniquely into words from the set $\{0 1^a : a \in \AA_1\}$, and so there exists a unique such couple $(k, x)$ (the zero coordinate of $x$ corresponds to the symbol $a \in \AA_1$ such that the word $0 1^a$ cover the coordinate $y_0$).
Hence $\btau_\WW$ and $\widetilde{\btau}_\WW$ are recognizable at level $0$.

For $n \ge 1$ the morphism $\tau_n$ is rotationally conjugate to the right permutative morphism $\widetilde{\tau}_n$ (see \Cref{ss:directive_sequences}).
Hence, the morphisms $\tau_n$ and $\widetilde{\tau}_n$ are fully recognizable for aperiodic points by \Cref{p:right_perm}.
We conclude the proof using \Cref{p:composition}.
\end{proof}

By combining \Cref{p:s_adic_ferenczi}, \Cref{l:recognizability} and the discussion above, we deduce the following.

\begin{theorem}\label{t:recognizable_sequence}
A subshift $(X, S)$ is a minimal Ferenczi subshift if and only if it is an $\SS$-adic subshift generated by a directive sequence $\btau_\WW$ as in \eqref{eq:tau} where the sequence $(a_{n,i} : n \ge 0, \ 0 \le i < q_n)$ is bounded.
\end{theorem}

\subsection{Some useful computations for Ferenczi subshifts}
In this section we show some useful relations between the parameters defining a minimal Ferenczi subshift $(X_\WW, S)$ defined by a sequence of generating words $\WW$ given by \eqref{eq:generating_words}.

Define 
\begin{equation}\label{eq:f_n}
f_n(a)
= \# \{0 \le i < q_{n-1} : a_{n-1,i} = a\}, \quad
a \in \AA_n, \quad
n \ge 1
\end{equation}
and let $f_n$ be the column vector $f_n = (f_n(a))_{a \in \AA_n}$.
For a vector $f$ in $\RR^\AA$ we denote $|f| = \sum_{a \in \AA} f(a)$.
Observe that for $n \ge 1$,
\begin{align}\label{eq:fn}
|f_n|
& = q_{n-1} \\
\label{eq:sum_fn}
\sum_{b \in \AA_n} f_n(b) \cdot b
& = \sum_{i = 0}^{q_{n-1}-1} a_{n-1, i}.
\end{align}

We now compute the height vectors associated with $\btau_\WW$ and give some estimates.
We recall Equation \eqref{eq:heights}:
\[
h_n(a)
= |\tau_{[0,n)}(a)|, \quad
a \in \AA_n, \quad
n \ge 0.
\]

\begin{lemma}\label{l:computation_heights}
Let $\WW = (w_n)_{n \ge 0}$ be a sequence of generating words and $\btau_\WW$ be the associated directive sequence given by \eqref{eq:tau}.
Then, the height vectors $(h_n)_{n \ge 0}$ associated with $\btau_\WW$ satisfy
\begin{equation}\label{eq:height_vectors}
h_n(a) = a + |w_{n-1}|, \quad
a \in \AA_n, \quad
n \ge 1.
\end{equation}
In particular, there exists $K \ge 1$ such that
\begin{equation}\label{eq:proportional_heights}
K^{-1} h_n(b)
\le h_n(a)
\le K h_n(b), \quad
a, b \in \AA_n, \quad
n \ge 0.
\end{equation}
Moreover, there exists a constant $L \ge 1$ such that
\begin{equation}\label{eq:heights_Q}
L^{-1} Q_{0,n-1} \le
h_n (a) \le
L Q_{0,n-1}, \quad
a \in \AA_n, \quad
n \ge 1
\end{equation}
\end{lemma}

\begin{proof}
The computation of $h_1$ is clear from the definition.
Assume that \eqref{eq:height_vectors} holds for $n \ge 1$.
For $a \in \AA_{n+1}$, by using \eqref{eq:fn}, \eqref{eq:sum_fn} and \eqref{eq:lengths}, we obtain
\begin{align*}
h_{n+1}(a)
&= \sum_{b \in \AA_n} h_n(b) M_{\tau_n}(b, a)
= h_n(a)(1 + f_n(a)) + \sum_{b \in \AA_n,\ b \not= a} h_n(b) f_n(b)\\
&= (a + |w_{n-1}|)(1 + f_n(a)) +
\sum_{b \in \AA_n,\ b \not= a} (b + |w_{n-1}|) f_n(b)\\
&= a + |w_{n-1}| + \sum_{b \in \AA_n} f_n(b) \cdot |w_{n-1}| + \sum_{b \in \AA_n} f_n(b) \cdot b\\
&= a + (q_{n-1} + 1) |w_{n-1}| + \sum_{i=0}^{q_{n-1} - 1} a_{n-1,i}\\
&= a + |w_n|,
\end{align*}
proving \eqref{eq:height_vectors} by induction.
The estimate \eqref{eq:proportional_heights} follows directly from \eqref{eq:height_vectors}.

By definition of the morphism $\tau_0$, there exists a constant $L \ge 1$ such that 
\[
L^{-1} |w| \le
|\tau_0(w) | \le
L |w|, \quad
w \in \AA_1^\ast.
\]
Let $a \in \AA_n$, $n \ge 1$.
By \eqref{eq:constant_length}, we have
\begin{align*}
h_n(a)
= |\tau_0 \circ \tau_{[1,n)}(a)|
\le L |\tau_{[1,n)}(a)|
= L \prod_{i=1}^{n-1} |\tau_i|
= L Q_{0,n-1}.
\end{align*}
Analogously, we obtain $L^{-1} Q_{0,n-1} \le h_n(a)$, thus obtaining \eqref{eq:heights_Q}.
\end{proof}

The composition matrices of the directive sequence $\btau_\WW$ can be computed as
\begin{equation}\label{eq:composition_matrices}
M_{\tau_0}
= \matriz{1 & \ldots & 1 \\ a_1 & \ldots & a_\ell}, \quad
M_{\tau_n}
= I_{n, n+1} + f_n \cdot \uu_n, \quad n \ge 1,
\end{equation}
where the matrix $I_{n, n+1}$ is given for each $a \in \AA_n$ and $b \in \AA_{n+1}$ by $I_{n, n+1}(a, b) = 1$ if $a = b$ and $0$ otherwise and $\uu_n$ is the row vector of ones in $\RR^{\AA_{n+1}}$.

Let $n_0 \in \NN$ be such that $\AA_n = \AA_\WW$ for all $n \ge n_0$, $I$ be the identity matrix in $\RR^{\AA_\WW}$ and $\uu$ be the row vector of ones in $\RR^{\AA_\WW}$.

\begin{lemma}\label{l:product}
Let $g_1, g_2, \ldots, g_n$ be column vectors indexed by a finite alphabet $\AA$.
Let
\[
A_i
= I + g_i \cdot \uu, \quad
1 \le i \le n,
\]
where $I$ is the identity in $\RR^\AA$ and $\uu$ is the row vector of ones in $\RR^\AA$.
Then
\[
A_1 A_2 \ldots A_n
= I + \left(\sum_{k = 1}^n \prod_{j = k+1}^n (1 + |g_j|) g_k\right) \cdot \uu, \quad
n \ge 1.
\]
and
\[
A_i^{-1}
= I - \frac{g_i}{|g_i| + 1} \cdot \uu, \quad
1 \le i \le n
\]
\end{lemma}

\begin{proof}
It is easy to check that the inverse of $A_i$ is as given.
The formula for the product $A_1 A_2 \ldots A_n$ is clearly true for $n = 1$.
Suppose that it is true for $n$ and let us show that it is true for $n+1$.
In fact,
\begin{align*}
A_1 \ldots A_n A_{n+1}
=& \left(I + \left(\sum_{k = 1}^n \prod_{j = k+1}^n (1 + |g_j|) g_k \right) \cdot \uu \right) 
\left(I + g_{n+1} \cdot \uu \right)\\
=& I + (1 + |g_{n+1}|) \left(\sum_{k = 1}^n \prod_{j = k+1}^n (1 + |g_j|) g_k \right) \cdot \uu + g_{n+1} \cdot \uu\\
=& I + \left(\sum_{k = 1}^n \prod_{j = k+1}^{n+1} (1 + |g_j|) g_k \right) \cdot \uu + g_{n+1} \cdot \uu\\
=& I + \left(\sum_{k = 1}^{n+1} \prod_{j = k+1}^{n+1} (1 + |g_j|) g_k\right)\cdot \uu.
\end{align*}
\end{proof}

By \Cref{l:product} and \eqref{eq:Qmn}, we have
\begin{equation}\label{eq:long_product}
M_{\tau_m} M_{\tau_{m+1}} \ldots M_{\tau_{n-1}}
= I + f_{m,n} \cdot \uu, \quad
n_0 \le m < n,
\end{equation}
where
\begin{equation}\label{eq:fmn}
f_{m,n}
= \sum_{k = m}^{n-1} Q_{k, n-1} f_k.
\end{equation}
Observe that
\begin{equation}\label{eq:fmn_Qmn}
|f_{m,n}| + 1
= Q_{m-1, n-1}.
\end{equation}

Thus, \Cref{l:product} implies
\begin{equation}\label{eq:inverse_product}
(M_{\tau_m} M_{\tau_{m+1}} \ldots M_{\tau_{n-1}})^{-1}
= I - \frac{f_{m,n}}{Q_{m-1, n-1}} \cdot \uu, \quad
n_0 \le m < n.
\end{equation}

\begin{example}
Let $a < b < c < d$ be positive integers.
Define a sequence $\WW = (w_n)_{n \ge 0}$ of generating words such that for infinitely many values of $n$
\[
w_{n+1}
= w_n 1^a w_n 1^b w_n \quad
\text{and} \quad
w_{n+1} = w_n 1^c w_n 1^d w_n.
\]
Hence $\AA_\WW = \{a, b, c, d\}$.
The directive sequence $\btau_\WW$ consists of two morphisms $\tau_{a,b}$ and $\tau_{c,d}$, each one occurring infinitely many times in $\btau_\WW$, defined, for $u \in \AA_\WW$, by
\begin{align*}
\tau_{a,b}(u)
&= b u a,\\
\tau_{c,d}(u)
&= d u c.
\end{align*}
The composition matrices indexed by $\AA_\WW$ are
\[
M_{\tau_{a,b}}
= \matriz{2 & 1 & 1 & 1 \\ 1 & 2 & 1 & 1 \\ 0 & 0 & 1 & 0 \\ 0 & 0 & 0 & 1}, \quad
M_{\tau_{c,d}}
= \matriz{1 & 0 & 0 & 0 \\ 0 & 1 & 0 & 0 \\ 1 & 1 & 2 & 1 \\ 1 & 1 & 1 & 2}.
\]
\end{example}

\section{Topological dynamical properties of minimal Ferenczi subshifts}\label{s:top_dyn}

In what follows $\WW$ is a standard sequence as given by \eqref{eq:generating_words} generating a minimal Ferenczi subshift.
Let $(X_\WW , S)$ be the subshift it generates and $\btau_\WW$ the associated directive sequence given by \eqref{eq:tau}.

\subsection{Unique ergodicity}

The unique ergodicity of Ferenczi subshifts is a folklore result \cite[Section 1.1.4]{Fer97}.
We provide a short proof below.

\begin{proposition}\label{p:uniq_ergodic}
The system $(X_\WW, S)$ is uniquely ergodic.
\end{proposition}

\begin{proof}
The directive sequence $\btau_\WW$ defines a sequence of measure vectors $(\mu_n)_{n \ge 0}$ given in \Cref{ss:invariant_measures}. 
By \eqref{eq:inv_measure}, to prove unique ergodicity of $(X_\WW, S)$ it is sufficient to prove that the vector $\mu_n$ is uniquely determined for infinitely many values of $n$.

\medskip

Recall the definition of $Q_{0,m}$ for $m \ge 1$ in \eqref{eq:Qmn}.
Let us consider the vectors $(t_m)_{m \ge 1}$ defined by $t_m = Q_{0,m-1} \mu_m$, $m \ge 1$.
By \eqref{eq:inv_measure} we have
\[
t_m
= Q_{0,m-1} \mu_m
= \frac{Q_{0,m}}{q_{m-1} + 1} (M_m \mu_{m+1})
= \frac{1}{q_{m-1}+1} M_m t_{m+1}, \quad
m \ge 1.
\]
Consequently, from equations \eqref{eq:long_product} and \eqref{eq:fmn}, 
\begin{align*}
t_m
&= \frac{1}{Q_{m-1,n-1}} M_m M_{m+1} \cdots M_{n-1} t_n
= \frac{1}{Q_{m-1,n-1}}
\left(I + \left( \sum_{k=m}^{n-1} Q_{k,n-1} f_k \right) \cdot \uu \right) t_n\\
&= Q_{0,m-1} \mu_n + \left(\sum_{k=m}^{n-1} \frac{f_k}{
Q_{m-1,k}}\right) \cdot |t_n|, \quad
n_0 \le m < n.
\end{align*}
It can be checked that $\sum_k \frac{f_k}{Q_{m-1,k}}$ converges, we define $v_m = \sum_{k=m}^\infty \frac{f_k}{Q_{m-1,k}}$.

We deduce $L^{-1} \le |t_n| \le L$ from \eqref{eq:heights_Q} and, since $\mu_n \to 0$ as $n \to +\infty$, there exists a sequence of nonnegative numbers $(\alpha_m)_{m \ge n_0}$ such that
\[
\mu_m
= \alpha_m v_m, \quad
m \ge n_0.
\]
We deduce $\alpha_{n_0} = \frac{1}{|P_{0,n_0} v_{n_0}|}$ from \eqref{eq:inv_measure}.
Again, from \eqref{eq:inv_measure} we obtain
\[
\alpha_m
= \frac{|v_{n_0}|}{|P_{0, n_0} v_{n_0}| |P_{n_0, m} v_{m}|}, \quad
m \ge n_0
\]
and finally
\[
\mu_m
= \left(
\frac{|v_{n_0}|}{|P_{0, n_0} v_{n_0}| |P_{n_0, m} v_{m}|}\right) v_m, \quad
m \ge n_0.
\]
This completes the proof.
\end{proof}

\subsection{Clean directive sequences}\label{ss:clean}

To go further in the study of Ferenczi subshifts we need the following notion inspired by the definition of \emph{clean Bratteli diagram} given in \cite[Section 5]{BDM10}, see also \cite[Theorem 3.3]{BKMS13}.

Let $\btau = (\tau_n : \AA_n^\ast \to \AA_n^\ast)_{n \ge 0}$ be a proper, primitive and recognizable directive sequence and let $\mu$ be an ergodic invariant probability measure of $(X_{\btau}, S)$.

We say that $\btau$ is \emph{clean} with respect to $\mu$ if:
\begin{enumerate}
    \item There exists $n_0 \in \NN$ such that $\AA_n = \AA_{n_0}$ for all $n \ge n_0$.
    Put $\AA = \AA_{n_0}$.
    
    \item There exists a constant $c > 0$ and $\AA_\mu \subseteq \AA$ such that
    \begin{equation}\label{eq:clean_ineq}
    \mu(\TT_n(a))
    \ge c, \quad
    n \ge n_0, \quad
    a \in \AA_\mu, \quad
    \text{and}
    \end{equation}
    \[
    \lim_{n \to +\infty}
    \mu(\TT_n(a))
    = 0, \quad
    a \in \AA \setminus \AA_\mu.
    \]
\end{enumerate}

We remark that we can always contract the directive sequence $\btau$ so that it becomes clean with respect to $\mu$.
If $\AA_\mu = \AA$, we say that $\btau$ is of \emph{exact finite rank}.

It is proven in \cite{BKMS13} that exact finite rank of $\btau$ implies that $(X_{\btau}, S)$ is uniquely ergodic.
However, the converse is not true, even for Ferenczi subshifts.

\begin{example}[Ferenczi subshift with non exact rank]
Consider a sequence of generating words $\WW$ with associated cutting parameters $(q_n)_{n \ge 0}$, as defined in \eqref{eq:generating_words}.
Suppose that $q_n \to +\infty$ as $n \to +\infty$ and that there exists a letter $a^\ast$ in $\AA_\WW$ such that $f_n(a^\ast) + 1 \le C$ for all large enough values of $n$ and some value $C > 0$.

If $\mu$ is the unique invariant probability measure of $(X_\WW, S)$, from \eqref{eq:prob_measure} we obtain
\begin{align*}
\mu(\TT_n(a^\ast))
&= h_n(a^\ast) \mu_n(a^\ast)
= h_n(a^\ast) \sum_{b \in \AA_\WW} M_n(a^\ast, b) \frac{\mu(\TT_{n+1}(b))}{h_{n+1}(b)}\\
&\le \frac{C h_n(a^\ast)}{\min_{b \in \AA_\WW} h_{n+1}(b)}
\le \frac{C K}{q_{n-1} + 1},
\end{align*}
where we used
\begin{align*}
h_{n+1}(b)
&= \sum_{c \in \AA_\WW} h_n(c) M_n(c, b)
\ge K^{-1} h_n(a^\ast)
\sum_{c \in \AA_\WW} M_n(c, b)\\
&= K^{-1} h_n(a^\ast)(q_{n-1}+1), \quad
b \in \AA_\WW.
\end{align*}
Therefore $\mu(\TT_n(a^\ast)) \to 0$ as $n \to +\infty$ and $\btau_\WW$ is not exact of finite rank.
\end{example}

A subshift $(X, S)$ is \emph{linearly recurrent} if it is minimal and there exists a constant $K > 0$ such that if $u \in \LL(X)$ and $w$ is a right return word to $u$ in $X$, then
\[
|w|
\le K |u|.
\]
We refer to \cites{Dur00, DP22} for more details on linearly recurrent shifts.
In \cite{BKMS13} it is shown that linearly recurrent subshifts have exact finite rank and that the converse is not true.
The example below shows that the converse is not true, even in the family of Ferenczi subshifts.

\begin{example}[Ferenczi subshift with exact finite rank that is not linearly recurrent]
Consider a sequence of generating words $\WW$ such that $d_\WW = 2$.
Let $\AA_\WW = \{a, b\}$ and define the morphism $\tau_n$ by
\[
\tau_n(a) = a^n b^{2n-1} a b \quad
\text{and} \quad
\tau_n(b) = a^n b^{2n-1} b b, \quad
n \ge 1.
\]

The composition matrix of $\tau_n$ indexed by $\AA_\WW$ is $M_{\tau_n} = \begin{psmallmatrix} n+1 & n\\ 2n & 2n+1 \end{psmallmatrix}$, and hence $\btau_\WW$ is of exact finite rank \cite[Proposition 5.7]{BKMS13}.
Observe that for all $n$ the word $\tau_{[0,n)} (a)^{n+1}$ belongs to the language of $X_\WW$. 
Hence the subshift $(X_\WW, S)$ is not linearly recurrent, see \cite[Theorem 24]{DHS99}.
\end{example}

In the sequel we characterize exact finite rank of $\btau_\WW$.

\begin{proposition}\label{p:exact_rank}
For $a \in \AA_\WW$, we have $\liminf_{m \to +\infty} \mu(\TT_m(a)) > 0$ if and only if
\begin{equation}\label{eq:exact_rank_tau}
\liminf_{m \to +\infty} \sum_{k=m}^\infty \frac{f_k (a)}{Q_{m-1,k}}
> 0.
\end{equation}
In particular, $\btau_\WW$ is of exact finite rank if and only if \eqref{eq:exact_rank_tau} holds for all $a \in \AA_\WW$.
\end{proposition}

\begin{proof}
Let $n_0 \in \NN$ be such that $\AA_{n_0} = \AA_\WW$.
Consider $m \ge n_0$ and $a \in \AA_\WW$.
Since $d_\WW \ge 2$, there exists $b \in \AA_\WW$ with $b \not= a$.
By \cite[Proposition 5.1]{BKMS13} one has
\[
\mu_m(a)
= \lim_{n \to +\infty } \frac{|\tau_{[m,n)} (b)|_a}{h_n(b)}.
\]
By \eqref{eq:long_product} and \eqref{eq:fmn}, since $b \not= a$ one gets
\[
\frac{|\tau_{[m,n)}(b)|_a}{h_n(b)}
= \frac{\sum_{k=m}^{n-1} Q_{k, n-1} f_k(a)}{h_n(b)}.
\]
Hence by \eqref{eq:prob_measure}
\begin{align*}
\mu(\TT_m(a))
&= h_m(a) \mu_m(a)
= \lim_{n \to +\infty}
\frac{h_m(a)}{h_n(b)} \sum_{k=m}^{n-1} Q_{k, n-1} f_k(a)\\
&= \lim_{n \to +\infty}
\frac{h_m(a)}{h_n(b)} Q_{m-1,n-1} \sum_{k=m}^{n-1} \frac{f_k(a)}{Q_{m-1, k}}.
\end{align*}
Using \eqref{eq:heights_Q}, there exists a constant $C \ge 1$ such that
\[
C^{-1}
\le \frac{h_m(a)}{h_n(b)} Q_{m-1,n-1}
\le C, \quad
a, b \in \AA_\WW, \quad
n_0 \le m < n.
\]
Therefore $\liminf_{m \to +\infty} \mu(\TT_m(a)) > 0$ if and only if
\[
\liminf_{m \to +\infty} \sum_{k=m}^\infty \frac{f_k (a)}{Q_{m-1,k}}
> 0,
\]
where it can be checked that $\sum_k f_k(a) / Q_{m-1, k}$ converges.
\end{proof}

\subsection{Toeplitz induced systems}

Let $(X_\WW, S)$ be a minimal Ferenczi subshift and $\btau_\WW = (\tau_n : \AA_{n+1}^\ast \to \AA_n^\ast)_{n \ge 0}$ be the directive sequence given by \eqref{eq:tau}.
Denote by $\btau_\WW' = (\tau_{n+1} : \AA_{n+2}^\ast \to \AA_{n+1}^\ast)_{n \ge 0}$ the shifted directive sequence of $\btau_\WW$ and let $U_\WW = \tau_0(X_{\btau_\WW'})$.
From \Cref{c:induced_sadic} the induced system $(U_\WW, S_{U_\WW})$ of $(X_\WW, S)$ on $U_\WW$ is topologically conjugate to the $\SS$-adic subshift $(X_{\btau_\WW'}, S)$.

Recall from \Cref{ss:directive_sequences} and \Cref{ss:recognizable_sequence} that each morphism $\tau_n$ has constant length and is rotationally conjugate to the right permutative morphism $\widetilde{\tau}_n$ for $n \ge 1$.
Therefore, $\btau_\WW'$ is recognizable and the subshift $(X_{\btau_\WW'}, S)$ is minimal and aperiodic.
Moreover, the associated sequence of incidence matrices $(M_n)_{n \ge 0}$ coincides with the sequence of composition matrices $(M_{\tau_{n+1}})_{n \ge 0}$ by \Cref{p:nested}.

We deduce that the latter has the \emph{equal path number property}, \emph{i.e.}, for each $n \ge 0$ the sum of each column of $M_n$ is constant.
This implies that $(X_{\btau_\WW'}, S)$ is topologically conjugate to a minimal Toeplitz subshift \cite[Theorem 8]{GJ00}.
We recall that a subshift $(X, S)$ with $X \subseteq \AA^\ZZ$ is Toeplitz if $X$ is the closure of the orbit $\{S^n x : n \in \ZZ\}$ for some sequence $x = (x_n)_{n \in \ZZ} \in \AA^\ZZ$ such that for all $n \in \ZZ$ there exists $p \in \NN$ with $x_n = x_{n + kp}$ for all $k \in \ZZ$.

We will prove that this Toeplitz subshift is \emph{mean equicontinuous}.
We recall that a topological dynamical system  $(X ,T)$ with a metric $d$ on $X$ is mean equicontinuous if for every $\epsilon > 0$ there exists $\delta > 0$ such that if $d(x, y) \le \delta$, then $\rho_{\text{b}}(x, y) \le \epsilon$.
Here, $\rho_{\text{b}}$ denotes the Besicovitch pseudometric given by
\[
\rho_{\text{b}}(x, y)
= \limsup_{n \to +\infty} \frac{1}{n} \sum_{k=0}^{n-1} d(T^k x, T^k y), \quad
x, y \in X.
\]

Let $(X, T)$ be a minimal system.
Denote by $(\Xeq, \Teq)$ to the maximal equicontinuous factor of $(X, T)$, by $\nu$ to its unique invariant probability measure and let $\pieq : X \to \Xeq$ be the corresponding factor map.

The system $(X, T)$ is mean equicontinuous if and only if it is uniquely ergodic (with unique invariant probability measure $\mu$) and $\pieq$ is a measurable isomorphism between the systems $(X, T, \mu)$ and $(\Xeq, \Teq, \nu)$ \cites{LTY15, DG16}. 
In particular, this implies that the system $(X, T, \mu)$ has discrete spectrum, \emph{i.e.}, there exists an orthonormal basis of $L^2(X, \mu)$ consisting of measurable eigenfunctions of $(X, T)$.
We refer to \cite{GJY21} for more details about mean equicontinuity.

\medskip

In what follows we will need the following definitions.
For a sequence of positive integers $(p_n)_{n \ge 0}$ such that $p_n$ divides $p_{n+1}$ for $n \ge 0$, the \emph{odometer} given by this sequence is the system $(\ZZ_{(p_n)_{n \ge 0}}, T)$, where
\[
\ZZ_{(p_n)_{n \ge 0}}
= \varprojlim \ZZ / p_n \ZZ
= \left\{ (x_n)_{n \ge 0} \in \prod_{n \ge 0} \ZZ / p_n \ZZ : x_{n+1} \equiv x_n \pmod{p_n},\ n \ge 0 \right\}
\]
and the map $T : \ZZ_{(p_n)_{n \ge 0}} \to \ZZ_{(p_n)_{n \ge 0}}$ is given by
\[
T((x_n)_{n \ge 0})
= (x_n + 1 \pmod{p_n})_{n \ge 0}.
\]
Let $\btau = (\tau_n : \AA_{n+1}^\ast \to \AA_n^\ast)_{n \ge 0}$ be a primitive, proper and recognizable directive sequence such that the morphism $\tau_n$ has constant length for each $n \ge 0$.
It is classical to show that the maximal equicontinuous factor of $(X_{\btau}, S)$ corresponds to the odometer $(\ZZ_{(|\tau_{[0,n)}|)_{n \ge 0}}, T)$ \cite{GJ00}.
The factor map $\pieq : X_{\btau} \to \ZZ_{(|\tau_{[0,n)}|)_{n \ge 0}}$ can be described as follows.
Let $x \in X_{\btau}$.
By recognizability of $\btau$, for every $n \ge 0$ there exists a letter $a_n(x)$ in $\AA_n$ and $k_n(x)$ with $0 \le k_n(x) < |\tau_{[0,n)}|$, uniquely determined, such that
\[
x
\in S^{k_n(x)} \tau_{[0,n)}([a_n(x)]).
\]
Then we define
\begin{equation}\label{eq:pieq}
\pieq(x)
= (k_n(x))_{n \ge 0}.
\end{equation}
It can be observed that $k_{n+1}(x) \equiv k_n(x) \pmod{|\tau_{[0,n)}|}$ for $x \in X_{\btau}$.

For a morphism $\tau : \AA^\ast \to \BB^\ast$ of constant length $|\tau|$, we say that it has a \emph{coincidence} at index $0 \le i < |\tau|$ if $\tau(a)_i = \tau(a')_i$ for every $a, a' \in \AA$.
The notion of coincidence has been used in \cite{Dek78} to characterize the discrete spectrum of constant length substitution systems.
See also \cite{Que87}.

\begin{proposition}
The system $(U_\WW, S_{U_\WW})$ is mean equicontinuous.
\end{proposition}

\begin{proof}
As was previously observed, the system $(U_\WW, S_{U_\WW})$ is topologically conjugate to the $\SS$-adic subshift $(X_{\btau_\WW'}, S)$.
Moreover, as in the proof of \Cref{p:uniq_ergodic}, this subshift is uniquely ergodic.
Denote by $\mu$ its unique invariant probability measure.

The directive sequence $\btau_\WW'$ is primitive, proper, recognizable and consists of morphisms of constant length.
Indeed, from \eqref{eq:tau} we have $|\tau_{n+1}| = q_n + 1$, $n \ge 0$.
Hence, the maximal equicontinuous factor of $(X_{\btau_\WW'}, S)$ is the odometer $(\ZZ_{(Q_{0,n})_{n \ge 0}}, T)$.
Denote by $\nu$ the unique invariant probability measure of this odometer and let $\pieq : X_{\btau_\WW'} \to \ZZ_{(Q_{0,n})_{n \ge 0}}$ be the factor map given by \eqref{eq:pieq}.
Denote by $(\TT_n')_{n \ge 0}$ the nested sequence of CKR partitions of $(X_{\btau_\WW'}, S)$ given by \eqref{eq:rep_sadic}.

For each $z = (z_n)_{n \ge 0}$ in $\ZZ_{(Q_{0,n})_{n \ge 0}}$, we write
\[
z_n = Q_{0, n-1} t_n(z) + r_n(z), \quad
0 \le r_n(z) < Q_{0, n-1}, \quad
0 \le t_n(z) < q_{n-1} + 1, \quad
n \ge 1.
\]
We define
\begin{align*}
C_n
&= \{0 \le i < (q_{n-1} + 1) : \text{$\tau_n$ has a coincidence at index $i$}\}\\
D_n
&= \{z \in \ZZ_{(Q_{0,n})_{n \ge 0}} : t_n(z) \notin C_n\}, \quad
n \ge 1.
\end{align*}

\begin{claim}
If a point $z = (z_n)_{n \ge 0}$ in $\ZZ_{(Q_{0,n})_{n \ge 0}}$ is such that $t_n(z)$ belongs to $C_n$ for infinitely many values of $n$, then $|\pieq^{-1}(\{z\})| = 1$.
\end{claim}

Indeed, let $x, y \in X_{\btau_\WW'}$ be such that $k_n(x) = k_n(y) = z_n$ for $n \ge 0$.
If $t_n(z)$ belongs to $C_n$, then there exists a letter $\ell(z)$ in $\AA_n$ such that $\tau_n(a_{n+1})_{t_n(z)} = \ell(z)$ for every $a_{n+1}$ in $\AA_{n+1}$.
Consequently, we have
\[
S^{z_n} \tau_{[1, n+1)}([a_{n+1}])
\subseteq S^{r_n(z)} \tau_{[1,n)}([\ell(z)]), \quad
a_{n+1} \in \AA_{n+1}.
\]
This implies that there exists infinitely many values of $n$ for which $x$ and $y$ belong to $S^{r_n(z)} \tau_{[1,n)}([\ell(z)])$.
Since $(\TT_n')_{n \ge 0}$ is a nested sequence, we deduce that $\diam(S^{r_n(z)} \tau_{[1,n)}([\ell(z)])) \to 0$ as $n \to +\infty$, and therefore $x = y$, proving the claim.

\medskip

From the claim, it follows that if we denote by $\mathcal{Z}$ the set of points in $\ZZ_{(Q_{0,n})_{n \ge 0}}$ that are not invertible under $\pieq$, then
\[
\mathcal{Z}
\subseteq \bigcup_{n \ge 0} \bigcap_{m \ge n} D_m.
\]
Observe that, from \eqref{eq:tau}, we obtain $\nu(D_m) = 1 - \frac{|C_m|}{q_m + 1} = \frac{1}{q_m + 1}$ for $m \ge 0$.
Thus
\[
\nu \Bigg( \bigcap_{m \ge n} D_m \Bigg)
= \prod_{m \ge n} \frac{1}{q_m + 1}
\le \prod_{m \ge n} \frac{1}{2}
= 0,
\]
and hence $\nu(\mathcal{Z}) = 0$.
This proves that $\pieq$ is a measurable isomorphism between $(X_{\btau_\WW'}, S, \mu)$ and $(\ZZ_{(Q_{0,n})_{n \ge 0}}, T, \nu)$, and concludes the proof.
\end{proof}

\subsection{Computation of the topological rank}

In this section we compute explicitly the topological rank of a minimal Ferenczi subshift $(X_\WW, S)$.
We refer to \Cref{s:kakutani_rokhlin} and \Cref{ss:dimension_groups} for the definitions.

For an abelian group $G$ we denote by $\rank G$ the \emph{rational rank} of $G$, \emph{i.e.},
\[
\rank_\QQ G
= \dim_\QQ G \otimes \QQ.
\]

\subsubsection{Basics on tensor products}

We will need very classical facts on tensor products between abelian groups and $\QQ$.
We recall what is needed to follow our arguments and refer to \cite{Bou62} for more details.

Let $G$ be an abelian group.
Then, it has a $\ZZ$-module structure and we can define the tensor product $G \otimes \QQ$.
Moreover, this product has the structure of a vector space over $\QQ$.
Elements in $G \otimes \QQ$ are linear combinations of the form
\[
\sum_{k=0}^n g_k \otimes q_k, \quad
n \in \NN, \quad
g_k \in G, \quad
q_k \in \QQ, \quad
0 \le k \le n.
\]
An element of the form $g \otimes q$ with $g \in G$ and $q \in \QQ$ is said to be a \emph{pure tensor}.

\begin{proposition}\label{p:basics_tensor}
Let $G$ be an abelian group and $(G_n)_{n \ge 0}$ be a sequence of abelian groups.
We have the following.
\begin{enumerate}
    \item $(\varinjlim G_n) \otimes \QQ$ and $\varinjlim (G_n \otimes \QQ)$ are isomorphic as vector spaces over $\QQ$.
    \item If $g \otimes q = 0$ in $G \otimes \QQ$, then $q = 0$ or $g$ is a torsion element in $G$.
\end{enumerate}
\end{proposition}

\subsubsection{Back to the computation of the topological rank}

The following lemma gives a lower bound for the topological rank of a minimal Cantor system.

\begin{lemma}\label{l:lower_bound}
Let $(X, T)$ be a minimal Cantor system.
Then
\begin{equation}\label{eq:rank_inequality}
\rank_\QQ H(X, T)
\le \rank (X, T).
\end{equation}
\end{lemma}

\begin{proof}
We can assume $\rank (X, T) < +\infty$.
Let $(\TT_n)_{n \ge 0}$ be any nested sequence of CKR partitions of $(X, T)$ such that $\liminf_{n \to +\infty} |\AA(\TT_n)| < +\infty$.
Denote by $(M_n)_{n \ge 0}$ to the sequence of incidence matrices of $(\TT_n)_{n \ge 0}$.
We can assume that $|\AA(\TT_n)| = p$, $n \ge 0$ for some $p \in \NN$.

By \Cref{p:dimension_group_CKR}, the dimension group $H(X, T)$ can be seen as the direct limit $\varinjlim \ZZ^{\AA(\TT_n)}$ with linear maps $M_n : \ZZ^{\AA(\TT_n)} \to \ZZ^{\AA(\TT_{n+1})}$, $n \ge 0$.
Define the linear maps $j_{n+1, n} : \ZZ^{\AA(\TT_n)} \otimes \QQ \to \ZZ^{\AA(\TT_{n+1})} \otimes \QQ$ on pure tensors by
\[
j_{n+1, n}(v \otimes q)
= v M_n \otimes q, \quad
v \in \ZZ^{\AA(\TT_n)}, \quad
q \in \QQ, \quad
n \ge 0
\]
and extend them by linearity to $\ZZ^{\AA(\TT_n)} \otimes \QQ$.
We consider $\varinjlim (\ZZ^{\AA(\TT_n)} \otimes \QQ)$ with linear maps $(j_{n+1,n})_{n \ge 0}$.

\Cref{p:basics_tensor} implies that $H(X, T) \otimes \QQ$ and $\varinjlim (\ZZ^{\AA(\TT_n)} \otimes \QQ)$ are isomorphic vector spaces over $\QQ$.
Each morphism $j_n : \ZZ^{\AA(\TT_n)} \otimes \QQ \to \varinjlim (\ZZ^{\AA(\TT_n)} \otimes \QQ)$ is linear and we have
\[
\dim_\QQ \ima j_n
\le \dim_\QQ \ima j_n + \dim_\QQ \ker j_n
= p.
\]
Since $\ima j_m \subseteq \ima j_n$ for $m < n$, there exists $N \in \NN$ such that $\ima j_m = \ima j_n$ for all $m, n \ge N$.
This, together with the fact that $\varinjlim (\ZZ^{\AA(\TT_n)} \otimes \QQ) = \bigcup_{n \ge 0} \ima j_n$, implies that $ \rank_\QQ H(X, T) \le p$.

Since the choice of the sequence $(\TT_n)_{n \ge 0}$ is arbitrary, we deduce \eqref{eq:rank_inequality}.
\end{proof}

The proof of the next lemma is essentially given in \cite[Theorem 4.1]{BCBD+21}.

\begin{proposition}\label{p:rational_rank}
Let $\btau = (\tau_n : \AA_{n+1}^\ast \to \AA_n^\ast)_{n \ge 0}$ be a primitive, proper and invertible directive sequence.
Let $d = |\AA_0|$.
Then
\[
\rank_\QQ H(X_{\btau}, S)
= \rank (X_{\btau}, S)
= d.
\]
\end{proposition}

\begin{proof}
By \Cref{p:aperiodic} we have that $(X_{\btau}, S)$ is an aperiodic subshift.
Let $(\TT_n)_{n \ge 0}$ be the sequence given by \eqref{eq:rep_sadic}.
Since $\btau$ is recognizable \cite[Theorem 3.1]{BSTY19}, by \Cref{p:nested} we have that $(\TT_n)_{n \ge 0}$ is a nested sequence of CKR partitions.
Hence, by \eqref{eq:rank} we have $\rank (X_{\btau}, S) \le d$.

Now we show that $H(X_{\btau}, S) \otimes \QQ$ is finite-dimensional.
Indeed, we prove that
\[
B
= \{[\chi_{[a]}] \otimes 1 : a \in \AA_0\}
\]
is a basis of $H(X_{\btau}, S) \otimes \QQ$, where $[\chi_{[a]}]$ denotes the class of the characteristic function of the cylinder $[a]$ in $H(X_{\btau}, S)$.
This will finish the proof since \eqref{eq:rank_inequality} implies
\[
d
= \rank_\QQ H(X_{\btau}, S)
\le \rank(X_{\btau}, S)
\le d.
\]

Following the same steps as in the proof of \cite[Theorem 4.1]{BCBD+21} and since the matrix $M_{\tau_{[0,n)}}^{-1}$ has rational entries, we deduce that $B$ spans $H(X_{\btau}, S) \otimes \QQ$.
Suppose that $\alpha = (\alpha_a)_{a \in \AA_0} \in \ZZ^{\AA_0}$ is such that
\[
\sum_{a \in \AA_0} \alpha_a [\chi_{[a]}] \otimes 1
= 0.
\]
The fact that $H(X_{\btau}, S)$ is a torsion-free abelian group \cite[Proposition 2.1.13]{DP22} and \Cref{p:basics_tensor} imply that $\sum_{a \in \AA_0} \alpha_a [\chi_{[a]}] = 0$ in $H(X_{\btau}, S)$.
Then, as in the proof of \cite[Theorem 4.1]{BCBD+21}, we obtain $\alpha = 0$ and $B$ is a basis.
\end{proof}

\Cref{p:rational_rank} and \eqref{eq:inverse_product} directly imply the following.

\begin{corollary}\label{c:rank}
Let $(X_\WW, S)$ be a minimal Ferenczi subshift.
Then
\[
\rank (X_\WW, S)
= d_\WW.
\]
\end{corollary}

\subsection{Computation of the dimension group}

We now compute the dimension group of a minimal Ferenczi subshift $(X_\WW, S)$.
Let us explain how we will proceed.

Let $\btau_\WW$ be the directive sequence associated with $(X_\WW, S)$ and $n_0 \in \NN$ be such that $\AA_n = \AA_\WW$ for $n \ge n_0$.
Define the directive sequence
\[
\widehat{\btau}_\WW
= (\tau_{n + n_0} : \AA_\WW^\ast \to \AA_\WW^\ast)_{n \ge 0}.
\]
By \Cref{l:product}, $\widehat{\btau}_\WW$ is invertible.
By \Cref{l:induced}, the $\SS$-adic subshift $(X_{\widehat{\btau}_\WW}, S)$ is topologically conjugate to an induced system of $(X_{\btau_\WW}, S)$ on some clopen set and, from \Cref{ss:dg_sadic}, the dimension group of $\widehat{\btau}_\WW$ is unital order isomorphic to $(\HH_\WW, \HH_\WW^+, \bm{1})$, where
\begin{align*}
\HH_\WW
&= \{y \in \RR^{\AA_\WW} : y M_{\tau_{n_0}} M_{\tau_{n_0 + 1}} \ldots M_{\tau_{n_0 + n - 1}} \in \ZZ^{\AA_\WW} \
\text{for all $n$ large enough}\};\\
\HH_\WW^+
&= \{y \in \RR^{\AA_\WW} : y M_{\tau_{n_0}} M_{\tau_{n_0 + 1}} \ldots M_{\tau_{n_0 + n - 1}} \in \ZZ_+^{\AA_\WW} \
\text{for all $n$ large enough}\};
\end{align*}
and $\bm{1}(a) = 1$ for $a \in \AA_\WW$.
Moreover, the dimension group of $(X_\WW, S)$ is unital order isomorphic to $(\HH_\WW, \HH_\WW^+, v_\WW)$, where $v_\WW = (|\tau_{[0, n_0)}(a)|)_{a \in \AA_\WW}$.

Recall the definition of the sequence $(q_n)_{n \ge 0}$ given in \eqref{eq:generating_words} and of $Q_{m,n}$ in \eqref{eq:Qmn}.

By \Cref{p:uniq_ergodic} and \Cref{l:induced}, the system $(X_{\widehat{\btau}_\WW}, S)$ is uniquely ergodic.
Denote by $\widehat{\mu}$ its unique invariant probability measure and define the column probability vector $\widehat{\bmu} \in \RR^{\AA_\WW}$ by
\[
\widehat{\bmu}(a)
= \widehat{\mu}([a]), \quad a \in \AA_\WW.
\]

As in the proof of \Cref{p:exact_rank}, by using \cite[Proposition 5.1]{BKMS13} we have
\begin{equation}\label{eq:mu_hat}
\widehat{\bmu}(a)
= \sum_{k = n_0}^\infty \frac{f_k(a)}{Q_{n_0 - 1, k}}, \quad a \in \AA_\WW.
\end{equation}

In order to describe the dimension group of $(X_\WW, S)$, we need to define, for a sequence of positive integers $(a_n)_{n \geq N}$, the following additive group
\[
\ZZ[(a_n)_{n \ge N}]
= \left\{ \frac{m}{a_N a_{N+1} \cdots a_n} : m \in \ZZ,\ n \ge N \right\}.
\]
If $a_n = a$ for all $n \ge N$, we write $\ZZ[1/a] = \ZZ[(a_n)_{n \ge N}]$.

Let $a' \in \AA_\WW$ be such that $a' = \min_{a \in \AA_\WW} a$.
Define $\BB_\WW = \AA_\WW \setminus \{a'\}$.

We see elements in $\RR^{\AA_\WW}$ as vectors in $\RR^{\BB_\WW} \times \RR$.
Define the column vector $\bm{z} \in \RR^{\AA_\WW}$ by
\[
\bm{z}(b)
= \widehat{\bmu}(b), \quad
b \in \BB_\WW \quad
\text{and} \quad
\bm{z}(a')
= 1.
\]

\begin{proposition}\label{p:computation_dg}
Let $(X_\WW, S)$ be a minimal Ferenczi subshift and $(q_n)_{n \ge 0}$ be the sequence given in \eqref{eq:generating_words}.
The dimension group $K^0(X_\WW, S)$ is unital order isomorphic to $(\GG_\WW, \GG_\WW^+, u_\WW)$, where
\begin{align*}
\GG_\WW
&= \ZZ^{\BB_\WW} \times \ZZ[(q_n + 1)_{n \ge n_0 - 1}];\\
\GG_\WW^+
&= \{x \in \GG_\WW : x \cdot \bm{z} > 0\} \cup \{0\};
\end{align*}
and $u_\WW$ is given by $u_\WW(b) = b - a'$, $b \in \BB_\WW$ and $u_\WW(a') = a' + |w_{n_0 - 1}|$.
\end{proposition}

\begin{proof}
We begin by proving the following.

\begin{claim}
We have
\[
\HH_\WW^+
= \{y\in \HH_\WW : y \cdot \widehat{\bmu} > 0\} \cup \{0\}.
\]
\end{claim}

Indeed, since the directive sequence $\widehat{\btau}$ is primitive, proper and recognizable, the measure $\widehat{\mu}$ is uniquely determined by the associated sequence of measure vectors $(\widehat{\mu}_n)_{n \ge 0}$ as defined in \Cref{ss:invariant_measures}.
For $a \in \AA_\WW$, denote by $e_a \in \ZZ^\AA$ the vector such that $e_a(b) = 1$ if $a = b$ and $0$ otherwise.

Let $\widehat{P}_n = M_{\tau_{n_0}} M_{\tau_{n_0 + 1}} \ldots M_{\tau_{n_0 + n - 1}}$ for $n > 0$.
By \eqref{eq:inv_measure}, we have
\[
\widehat{\bmu}
= \widehat{P}_n \mu_n, \quad
n > 0.
\]
If $y$ belongs to $\HH_\WW^+ \setminus \{0\}$ and $n > 0$ is such that $y \widehat{P}_n$ is in $\ZZ_+^{\AA_\WW}$, then 
\[
y \cdot \widehat{\bmu}
= y \cdot \widehat{P}_n \mu_n
= (y \widehat{P}_n)\cdot \mu_n > 0.
\]

Now, let $y \in \HH_\WW$ with $y \cdot \widehat{\bmu} > 0$.
By contradiction, if $y$ is not in $\HH_\WW^+$, there exists $N \in \NN$ and a sequence $(a_n)_{n \ge N}$ such that $a_n$ belongs to $\AA_\WW$ and $(y \widehat{P}_n) \cdot e_{a_n} \le -1$ for all $n \ge N$.
Hence, there exists $a \in \AA_\WW$ and a sequence $(n_k)_{k \ge 0}$ such that $n_k \ge N$ and $(y \widehat{P}_{n_k})\cdot e_a \le -1$ for $k \ge 0$.
Let $\widehat{\bm{\nu}}$ be a limit point of the sequence of probability vectors $(\widehat{P}_{n_k} e_a / |\widehat{P}_{n_k} e_a|)_{k \ge 0}$.

By unique ergodicity of $(X_{\widehat{\btau}}, S)$, we deduce $\widehat{\bm{\nu}} = \widehat{\bmu}$.
Finally, up to passing to a subsequence,
\[
0 < y \cdot \widehat{\bmu}
= \lim_{k \to +\infty}
\frac{(y \widehat{P}_{n_k}) \cdot e_a}{|\widehat{P}_{n_k} e_a|}
\le 0,
\]
a contradiction.
This proves the claim.

Let $y \in \HH_\WW$. 
There exists $n \ge n_0$ with $y M_{\tau_{n_0}} M_{\tau_{n_0 + 1}} \ldots M_{\tau_{n_0 + n - 1}} \in \ZZ^{\AA_\WW}$.
Recall that $\uu$ is the row vector of ones in $\RR^{\AA_\WW}$.
By \eqref{eq:long_product}, we see that
\[
y M_{\tau_{n_0}} M_{\tau_{n_0 + 1}} \ldots M_{\tau_{n_0 + n - 1}}
= y + (y \cdot f_{n_0, n_0 + n}) \uu
\in \ZZ^{\AA_\WW},
\]
and hence $y(b) - y(a')$ belongs to $\ZZ$, $b \in \BB_\WW$.
Observe that 
\[
\uu M_{\tau_{n_0}} M_{\tau_{n_0 + 1}} \ldots M_{\tau_{n_0 + n - 1}}
= (|f_{n_0, n_0 + n}| + 1) \uu.
\]
Moreover, from \eqref{eq:Qmn} and \eqref{eq:fmn_Qmn}, we have \[
|f_{n_0, n_0 + n}| + 1
= (q_{n_0 - 1} + 1) (q_{n_0} + 1) \ldots (q_{n_0 + n - 2} + 1),
\]
so $y(a')$ belongs to $\ZZ[(q_n + 1)_{n \ge n_0 - 1}]$.

These two observations allow us to define the following group isomorphism
\begin{align*}
\psi : \HH_\WW
&\to \ZZ^{\BB_\WW} \times \ZZ[(q_n + 1)_{n \ge n_0 - 1}]\\
y &\mapsto (y', y(a'))
\end{align*}
where $y'(b) = y(b) - y(a')$ for $b \in \BB_\WW$.
Moreover, for $y \in \HH_\WW$ we have $y \cdot \widehat{\bmu} = \psi(y) \cdot \bm{z}$ and hence
\[
\psi(\HH_\WW^+)
= \{x \in \GG_\WW : x \cdot \bm{z} > 0\} \cup \{0\}.
\]
From \Cref{l:computation_heights}, we obtain $\psi(v_\WW) = u_\WW$.
This completes the proof.
\end{proof}

\subsection{Zoology of dimension groups of Ferenczi type}

We now characterize the dimension groups that can be obtained from minimal Ferenczi subshifts.
For this, we need to recall the following well-known fact about numeration systems.

\subsubsection{Facts about numeration systems}

Let $(p_k)_{k \ge 0}$ be a sequence of positive integers with $p_k \ge 2$, $k \ge 0$.
Then, for every real number $x$ with $0 \le x \le 1$, there exists a sequence $(f_k)_{k \ge 1}$ such that $0 \le f_k \le p_{k-1}$ for $k \ge 1$ and
\[
x
= \sum_{k = 1}^\infty \frac{f_k}{p_0 p_1 \ldots p_{k-1}}.
\]

We say that $(f_k)_{k \ge 1}$ is the \emph{expansion} of $x$ in the \emph{base} $(p_k)_{k \ge 0}$.

\subsubsection{Ferenczi type dimension groups}

Let $\BB$ be a nonempty alphabet.
We define
\[
U^\BB
= \{u \in \ZZ_{>0}^\BB : u(b) \not= u(b')\ \text{for}\ b, b' \in \BB\}.
\]
Observe that the unit $u_\WW$ in \Cref{p:computation_dg} belongs to $U^{\BB_\WW} \times \ZZ_{>0}$ since all elements in $\AA_\WW$ are distinct.
Define
\[
\Delta^\BB
= \{z \in \RR_{>0}^\BB : \textstyle\sum_{b \in \BB} z(b) < 1
\}
\]
and let $(r_n)_{n \ge 0}$ be a sequence of integers with $r_n \ge 2$.

We say that a dimension group $(\GG, \GG^+, u)$ is of \emph{Ferenczi type} if there exist a nonempty alphabet $\BB$, a sequence $(r_n)_{n \ge 0}$ as before and $\bm{z} \in \Delta^\BB \times \{1\}$ such that
\begin{align*}
\GG
&= \ZZ^\BB \times \ZZ[(r_n + 1)_{n \ge 0}]; \\
\GG^+
&= \{x \in \GG : x \cdot \bm{z} > 0\} \cup \{0\}; \quad
\text{and} \quad \\
u &\in U^\BB \times \ZZ_{>0}.
\end{align*}

\Cref{p:computation_dg} shows that the dimension group of a minimal Ferenczi subshift is of Ferenczi type.
Conversely, let $(\GG, \GG^+, u)$ be a dimension group of Ferenczi type given by $\BB$, $(r_n)_{n \ge 0}$ and $\bm{z}$.
Write $u = (v, w)$, where $v \in U^\BB$ and $w \in \ZZ_{>0}$.

Let $a' = w - 1$ and $s(b) = a' + v(b)$ for $b \in \BB$.
Observe that $s(b) > a'$ and $s(b) \not= s(b')$ for $b, b' \in \BB$.
Define any sequence of generating words $\WW$ such that:
\begin{enumerate}
    \item $n_0 = 1$ and $\AA_\WW = \{s(b) : b \in \BB\} \cup \{a'\}$;
    \item $q_n = r_n$ for $n \ge 0$; and
    \item for $b \in \BB_\WW$, let $(f_k(s(b)))_{k \ge 1}$ be the expansion of $\bm{z}(b)$ in the base $(q_k + 1)_{k \ge 0}$, \emph{i.e.},
    \[
    \bm{z}(b)
    = \sum_{k=1}^\infty \frac{f_k(s(b))}{(q_0+1) (q_1+1) \ldots (q_{k-1}+1)}, \quad
    b \in \BB.
    \]
\end{enumerate}

From Equation \eqref{eq:mu_hat}, we have thus proved the following.

\begin{corollary}\label{c:strong_ferenczi}
A dimension group $\mathcal{K} = (\GG, \GG^+, u)$ is of Ferenczi type if and only if there exists a minimal Ferenczi subshift $(X_\WW, S)$ such that $\mathcal{K}$ is unital order isomorphic to $K^0(X_\WW, S)$.
\end{corollary}

\begin{example}
\noindent
\begin{enumerate}
    \item The \emph{Chacon subshift} is defined by the sequence of generating words $\WW$ which satisfies
    \[
    w_{n+1}
    = w_n w_n 1 w_n, \quad
    n \ge 0.
    \]
    \Cref{p:computation_dg} shows that the dimension group of the Chacon subshift is
    \[
    (\ZZ \times \ZZ[1/3], \quad
    \{(x,y) \in \ZZ \times \ZZ[1/3] : x + 2y > 0\} \cup \{(0,0)\}, \quad 
    (1,1)).
    \]
    This dimension group is unital order isomorphic to
    \[
    (\ZZ \times \ZZ[1/3], \quad
    \ZZ \times \ZZ_+[1/3], \quad
    (1,1)).
    \]
    
    \item The \emph{Thue--Morse subshift} is the subshift generated by the constant directive sequence $\btau = (\tau, \tau, \ldots)$, where the morphism $\tau : \{a, b\}^\ast \to \{a, b\}^\ast$ is given by $\tau(a) = a b$ and $\tau(b) = b a$.
    Its dimension group is
    \[
    (\ZZ \times \ZZ[1/2], \quad
    \{(x,y) \in \ZZ \times \ZZ[1/2] : -x + 3y > 0\} \cup \{(0,0)\}, \quad 
    (0,1)),
    \]
    see \cite[Exemple 4.6.11]{DP22}.
    
    We claim that the Thue--Morse subshift is not strongly orbit equivalent to a minimal Ferenczi subshift.
    Indeed, by \Cref{c:strong_ferenczi}, suppose that there exists a nonempty alphabet $\BB$, a sequence $(r_n)_{n \ge 0}$, an order unit $u \in U^\BB \times \ZZ_{>0}$ and a isomorphism
    \[
    \psi :
    \ZZ \times \ZZ[1/2]
    \to \ZZ^\BB \times \ZZ[(r_n + 1)_{n \ge 0}]
    \]
    such that $\psi(0,1) = u$.
    
    The existence of $\psi$ ensures the existence of an isomorphism between $(\ZZ \times \ZZ[1/2]) \otimes \QQ$ and $(\ZZ^\BB \times \ZZ[(r_n + 1)_{n \ge 0}]) \otimes \QQ$, and thus $|\BB| = 1$.
    
    For a prime number $p$, denote by $v_p(\cdot)$ the $p$-adic valuation.
    For an integer sequence $(a_n)_{n \ge 0}$, the sequence $(v_p(a_0 \ldots a_n))_{n \ge 0}$ is increasing, and hence it is eventually constant or tends to $+\infty$.
    We denote by $v_p((a_n)_{n \ge 0})$ the eventually constant value of it (either finite or $+\infty$).
    
    It is easy to show that $v_2((r_n + 1)_{n \ge 0}) = +\infty$ and $v_p((r_n + 1)_{n \ge 0}) = 0$ for $p \not= 2$, so we can suppose $\ZZ[(r_n + 1)_{n \ge 0}] = \ZZ[1/2]$.
    
    Write $\psi(0,1) = (m, w)$, $m \in \ZZ$, $w \in \ZZ[1/2]$ and $\psi(0, 1/2^n) = (m_n, w_n)$, $m_n \in \ZZ$, $w_n \in \ZZ[1/2]$.
    Then 
    \[
    (2^n m_n, 2^n w_n)
    = (m,w), \quad
    n \in \NN.
    \]
    In particular, $2^n$ divides $m$ for all $n \ge 0$, hence $m = 0$.
    If $\psi(1,0) = (d,v)$ for $d \in \ZZ$ and $v \in \ZZ[1/2]$, we obtain
    \[
    \psi(s,t)
    = (ds, sv + tw), \quad
    s \in \ZZ, \quad
    t \in \ZZ[1/2].
    \]
    We have $u = \psi(0,1) = (0, w)$, but $0$ does not belong to $U^\BB$.
    This shows that the Thue--Morse subshift is not strongly orbit equivalent to any minimal Ferenczi subshift.
\end{enumerate}
\end{example}

\subsection{Comments on orbit equivalence}

In this section we characterize the orbit equivalence class of minimal Ferenczi subshifts.
With this purpose, we compute explicitly the dimension group $K^0(X_\WW, S) / \Inf{K^0(X_\WW, S)}$ of a minimal Ferenczi subshift $(X_\WW, S)$.
Recall the definition of $\GG_\WW$, $\bm{z}$ and $u_\WW$ given in \Cref{p:computation_dg}.

\begin{proposition}\label{p:orbit_eq}
Let $(X_\WW, S)$ be a minimal Ferenczi subshift.
Let $\widetilde{\bm{z}}$ be the unique vector collinear to $\bm{z}$ and such that $u_\WW \cdot \widetilde{\bm{z}} = 1$.
Define
\[
\JJ_\WW
= \{x \cdot \widetilde{\bm{z}} : x \in \GG_\WW\}.
\]
Then, the dimension group $K^0(X_\WW, S) / \Inf{K^0(X_\WW, S)}$ is unital order isomorphic to
\[
(\JJ_\WW, \quad
\{y \in \JJ_\WW : y \ge 0\}, \quad
1).
\]
\end{proposition}

\begin{proof}
From \Cref{p:computation_dg}, we see that $\Inf \GG_\WW = \{x \in \GG_\WW : x \cdot \widetilde{\bm{z}} = 0\}$.
It is straightforward to check that the map
\begin{align*}
\GG_\WW / \Inf{\GG_\WW}
&\to \JJ_\WW \\
[x]
&\mapsto x \cdot \widetilde{\bm{z}}
\end{align*}
is an isomorphism between the dimension groups $\GG_\WW / \Inf{\GG_\WW}$ and $\JJ_\WW$.
Moreover, this map sends the induced image of $\GG_\WW^+$ in $\GG_\WW / \Inf{\GG_\WW}$ to $\{x \in \JJ_\WW : x \ge 0\}$ and $[u_\WW]$ to $1$ since $u_\WW \cdot \widetilde{\bm{z}} = 1$.
\end{proof}

In particular, observe that if $\bm{z}$ has rationally independent entries, then the strong orbit equivalence class of $(X_\WW, S)$ coincides with the orbit equivalence class.

One can check that $\widetilde{\bm{z}}
= c \bm{z}$, where
\[
c = 
\lim_{n \to +\infty} \frac{Q_{n_0 - 1, n_0 + n - 1}}{|w_{n_0 + n - 1}|}.
\]

\subsection{Continuous eigenvalues}\label{ss:continuous_eigs}

In this section and the following we give different proofs for results established in \cite{GZ19} concerning continuous eigenvalues, topological weak mixing and topological mixing of minimal Ferenczi subshifts.
These proofs can have their own interest.

In the following, we use some general results for the existence of continuous eigenvalues in the context of minimal Cantor systems stated in \cite{DFM19}.

\begin{proposition}\label{p:rational_cont_eigs}
The system $(X_\WW, S)$ has no continuous irrational eigenvalues.
Moreover, the complex value $\lambda = \exp(2 \pi i p/q)$ with $p/q$ a rational number is a continuous eigenvalue of $(X_\WW, S)$ if and only if there exists $n \ge 0$ such that $q$ divides $|w_n| + a_{m,i}$ for all $m \ge n$ and $0 \le i < q_m$.
\end{proposition}

\begin{proof}
Suppose that $\lambda = \exp(2 \pi i \alpha)$ is a continuous eigenvalue of $(X_\WW,S)$, where $\alpha$ is an irrational number.
Recall the definition of the height vectors in \eqref{eq:heights}.
From \cite[Theorem 1]{DFM19}, necessarily we have
\[
\lim_{n \to +\infty} \inttt{\alpha h_n(a)}
= 0, \quad
a \in \AA_\WW.
\]
From \eqref{eq:height_vectors}, we deduce $\inttt{\alpha(a - b)} = 0$ for all $a, b \in \AA_\WW$.
Since $d_\WW \ge 2$, this implies that $\alpha$ is a rational number, which contradicts our assumption.

Now, from \cite[Corollary 6]{DFM19}, the value $\lambda = \exp(2 \pi i p/q)$ is a continuous eigenvalue of the system if and only if there exists $n \ge 0$ such that $q$ divides all the coordinates of the height vector $h_{n+1}$.
From the definition of $\AA_{n+1}$ and by \eqref{eq:height_vectors}, this is equivalent to say that $q$ divides $|w_n| + a_{m,i}$ for all $m \ge n$ and $0 \le i < q_m$.
\end{proof}

\begin{corollary}[{\cite[Theorem 1.1]{GZ19}}]\label{c:top_wm}
A minimal Ferenczi subshift $(X_\WW, S)$ is topologically weakly-mixing if and only if for all integer $q > 1$ and all $n \ge 0$ there exists $m \ge n$ and $0 \le i < q_m$ such that $q$ does not divide $|w_n| + a_{m,i}$.
\end{corollary}

\Cref{p:rational_cont_eigs} allows us to compute the \emph{maximal equicontinuous factor} of the subshift $(X_\WW, S)$.
We refer to \cite[Section 2.6]{Kur03} for the general definition of the maximal equicontinuous factor.

For an integer $q > 1$, we denote by $X_q$ the topological dynamical system
\[
X_q
= (\ZZ / q\ZZ, +1 \pmod{q}).
\]
Since $(X_\WW, S)$ is minimal, the complex value $\lambda = \exp(2 \pi i p/q)$ with $p, q$ coprime integers is a rational continuous eigenvalue of $(X_\WW, S)$ if and only if $X_q$ is a topological factor of $(X_\WW, S)$.

\begin{corollary}[{\cite[Theorem 1.5]{GZ19}}]\label{c:max_eq_factor}
Let $(X_\WW, S)$ be a minimal Ferenczi subshift.
Then, there exists a maximal integer $q_{\max}$ such that $X_{q_{\max}}$ is a topological factor of $(X_\WW, S)$.
Moreover, this factor corresponds to the maximal equicontinuous factor of $(X_\WW, S)$.
\end{corollary}

\begin{proof}
If $X_q$ is a topological factor of $(X_\WW, S)$, then $\lambda = \exp(2 \pi i / q)$ is a rational continuous eigenvalue of $(X_\WW, S)$.
By \Cref{p:rational_cont_eigs}, $q$ must divide $a - b$ for all $a, b \in \AA_\WW$.
In particular, $q \le \max_{a, b \in \AA_\WW} |a - b|$.
We deduce that there exist a finite set $\mathcal{C}$ of integers $q$ such that $X_q$ is a topological factor of $(X_\WW, S)$.

We let $q_{\max} = \max \mathcal{C}$.
It is straightforward to see that $X_{q_{\max}}$ has the same group of continuous eigenvalues as $(X_\WW, S)$ and that $q_{\max}$ is the maximal value of an integer $q$ such that $X_q$ is a topological factor of $(X_\WW, S)$.
From \cite[Theorem 2.56]{Kur03}, we conclude that $X_{q_{\max}}$ is the maximal equicontinuous factor of $(X_\WW, S)$.
\end{proof}

\begin{example}\label{ex:top_eigs}
\noindent
\begin{enumerate}
    \item Suppose that $\gcd(a - b : a, b \in \AA_\WW) = 1$, where $\gcd$ stands for the greatest common divisor.
    From \Cref{p:rational_cont_eigs}, we deduce that the system $(X_\WW, S)$ is topologically weakly-mixing.
    
    \item Let $q \in \NN$.
    Suppose that $q$ divides $a_i$ for $1 \le i \le \ell$ and also $q$ divides $q_0 + 1$.
    From \Cref{p:rational_cont_eigs}, we deduce that $\lambda = 1/q$ is a continuous eigenvalue of $(X_\WW, S)$.
\end{enumerate}
\end{example}

\subsection{Topological mixing}

We recall that a topological dynamical system $(X, T)$ is said to be \emph{topologically mixing} if for any nonempty open sets $U, V \subseteq X$, there exists $N \in \NN$ such that
\[
T^n U \cap V \not= \emptyset, \quad
n \ge N.
\]
It is classical to observe that topological mixing implies topological weak mixing.

We will prove that minimal Ferenczi subshifts cannot be topologically mixing.

For a subshift $(X, S)$ with $X \subseteq \{0,1\}^\ZZ$, we define the quantities 
\[
a(n)
= \min_{w \in \LL_n(X)} |w|_0 \quad
\text{and} \quad
b(n)
= \max_{w \in \LL_n(X)} |w|_0, \quad
n \ge 1.
\]

\begin{lemma}\label{l:nec_mixing}
Suppose that $(X, S)$ is minimal and topologically mixing.
Then 
\[
\lim_{n \to +\infty} b(n) - a(n)
= +\infty.
\]
\end{lemma}

\begin{proof}
Following \cite[Proposition 3.2]{KSS05}, if $(X, S)$ is topologically mixing we have
\begin{equation}\label{eq:KSS}
\liminf_{n \to +\infty} b(n) - a(n)
= \sup_{n \ge 1} b(n) - a(n).
\end{equation}

\begin{claim}
For any invariant ergodic probability measure $\mu$ of $(X, S)$ we have
\[
a(n)
\le n \mu([0])
\le b(n), \quad
n \ge 1.
\]
\end{claim}
 
Indeed, it is easy to see that the sequence $(a(n) / n)_{n \ge 1}$ is superadditive, \emph{i.e.}, $a(m+n) \ge a(m) + a(n)$ for all $m, n \ge 1$ and that the sequence $(b(n) / n)_{n \ge 1}$ is subadditive, \emph{i.e.}, $b(m+n) \le b(m) + b(n)$ for $m, n \ge 1$.
In particular, we deduce 
\[
\lim_{n \to +\infty} \frac{a(n)}{n}
= \sup_{n \ge 1} \frac{a(n)}{n} \quad
\text{and} \quad
\lim_{n \to +\infty} \frac{b(n)}{n}
= \inf_{n \ge 1} \frac{b(n)}{n}.
\]

Observe that
\[
\frac{a(n)}{n}
\le \frac{1}{n}
\#\{0 \le k < n : x_k = 0\}
= \frac{1}{n} \sum_{k=0}^{n-1} \chi_{[0]}(S^k x), \quad
x \in X, \quad
n \ge 1.
\]
By Birkhoff's theorem, there exists $x \in X$ such that
\[
\lim_{m \to +\infty} \frac{1}{m} \sum_{k=0}^{m-1} \chi_{[0]}(S^k x)
= \mu([0]).
\]
Hence, since $\frac{a(n)}{n} \le \sup_{m \ge 1} \frac{a(m)}{m}$, we obtain $\frac{a(n)}{n} \le \mu([0])$.
Analogously, we obtain $\mu([0]) \le \frac{b(n)}{n}$.
We obtain
\begin{equation}\label{eq:Gottschalk_Hedlund}
\Big\vert
\sum_{k=0}^{n-1} \chi_{[0]}(S^k x) - n \mu([0])
\Big\vert
\le b(n) - a(n), \quad
x \in X, \quad
n \ge 1.
\end{equation}
Let $f = \chi_{[0]} - \mu([0])$.
If $\limsup_{n \to +\infty} b(n) - a(n) \not= +\infty$, by \eqref{eq:Gottschalk_Hedlund} we deduce that there exists a constant $C > 0$ such that $\vert \sum_{k=0}^{n-1} f(S^k x) \vert \le C$, for all $x \in X$, $n \ge 1$.
Gottschalk--Hedlund's theorem then implies that $f = g - g \circ S$ for some continuous map $g : X \to \RR$.
In particular,
\[
\exp(2 \pi i g \circ S)
= \exp(2 \pi i \mu([0])) \exp(2 \pi i g),
\]
\emph{i.e.}, $\exp(2 \pi i \mu([0]))$ is a nontrivial continuous eigenvalue of $(X, S)$.
This contradicts the fact that $(X, S)$ is topologically weakly-mixing.

Finally, we deduce $\limsup_{n \to +\infty} b(n) - a(n) = +\infty$ and, together with \eqref{eq:KSS}, we obtain $\lim_{n \to +\infty} b(n) - a(n) = +\infty$.
\end{proof}

\Cref{c:top_wm} characterizes topological weak mixing of minimal Ferenczi subshifts.
On the other hand, we now give an alternative proof of the following result.

\begin{proposition}[{\cite[Theorem 1.3]{GZ19}}]
Minimal Ferenczi subshifts are not topologically mixing.
\end{proposition}

\begin{proof}
Let $(X_\WW, S)$ be a minimal Ferenczi subshift and consider the increasing sequence $(n_k)_{k \ge 0}$ given by $n_k = |w_k|$, $k \ge 0$.
Recall that the sequence $(a_{n,i} : n \ge 0,\ 0 \le i < q_n)$ takes finitely many values $a_1, a_2, \ldots, a_\ell$.

Let $k \ge 0$ and $u$ in $\LL_{n_k}(X_\WW)$.
Then $u$ must be a factor of $w_j$ for some $j > k$.
By induction, it is easy to see that we can write
\[
w_j
= w_k 1^{\alpha_0} w_k 1^{\alpha_1} \ldots w_k 1^{\alpha_{l-1}} w_k
\]
for some $l \in \NN$ and where $\alpha_i$ belongs to $\{a_1, \ldots, a_\ell\}$ for $0 \le i < l$.
From this we deduce that $b(n_k) = |w_k|_0$ and that there exists $\alpha \in \{a_1, a_2, \ldots, a_\ell\}$, some prefix (resp. suffix) $p$ (resp. $s$) of the word $w_k$ that satisfy $|p| + |s| = |w_k| - \alpha$ and $a(n_k) = |p|_0 + |s|_0$.
We obtain
\[
b(n_k) - a(n_k)
\le \alpha \le \max_{1 \le i \le \ell} a_i, \quad
k \ge 0.
\]
In particular $\lim_{n \to +\infty} b(n) - a(n) \not= +\infty$, and we conclude by \Cref{l:nec_mixing}.
\end{proof}

\begin{remark}
\Cref{p:uniq_ergodic} implies that minimal Ferenczi subshifts are not mixing with respect to its unique invariant probability measure.
\end{remark}

\subsection{Asymptotic classes and automorphism group}

Let $(X, T)$ be a topological dynamical system and $d : X \times X \to \RR$ be a metric on $X$.
We say that two points $x,y \in X$ are \emph{asymptotic} if
\[
\lim_{n \to +\infty} d(T^n x, T^n y)
= 0.
\]
Nontrivial asymptotic pairs of points may not exist in an arbitrary topological dynamical system, but they always exist in the context of nonempty aperiodic subshifts \cite[Chapter 1]{Aus88}.

We define the relation $\sim$ in $X$ as follows: $x \sim y$ if $x$ is asymptotic to $T^k y$ for some $k\in \ZZ$.
This defines an equivalence relation.
An equivalence class for $\sim$ that is not the orbit of a single point is called an \emph{asymptotic class}.

An \emph{automorphism} of a topological dynamical system $(X, T)$ is a homeomorphism $\phi : X \to X$ such that
\[
\phi \circ T
= T \circ \phi.
\]
We denote by $\Aut(X, T)$ the group of automorphism of $(X, T)$ and by $\langle T \rangle$ the subgroup of $\Aut(X, T)$ generated by integer powers of $T$.

For a minimal Ferenczi subshift we show that there exists a unique asymptotic class.
We first need the following lemma, for which we recall the definition of cutting points given in \Cref{ss:recognizability}.

\begin{lemma}\label{l:asymptotic}
Let $\tau : \AA^\ast \to \BB^\ast$ be a nonerasing morphism, $X \subseteq \AA^\ZZ$ be a subshift and $Y = \bigcup_{k \in \ZZ} S^k \tau(X)$.
Assume that $\tau$ is recognizable in $X$ and that if $a$ and $b$ are two distinct letters in $\AA$, then $\tau(a)$ is not a suffix of $\tau(b)$.
Let $y, y'$ in $Y$ be such that $y_0 \not= y'_0$ and $y_{(0, +\infty)} = y'_{(0, +\infty)}$.
Suppose that $(k, x)$ and $(k', x')$ are the unique centered $\tau$-representations of $y$ and $y'$ in $X$, respectively.
Then
\[
\mathcal{C}_\tau^+(k, x)
= \mathcal{C}_\tau^+(k', x'), \quad
x_0 \not= x'_0 \quad
\text{and} \quad
x_{(0, +\infty)} = x'_{(0, +\infty)}.
\]
\end{lemma}

\begin{proof}
We begin by proving the following.

\begin{claim}
There exist infinitely many pairs $(\ell, \ell')$ with $\ell, \ell' \ge 0$ such that
\[
C_\tau^\ell(k, x)
= C_\tau^{\ell'}(k', x').
\]
\end{claim}

Indeed, by \cite[Lemma 3.2]{DDMP21} there exists a constant $R > 0$ such that if $(k, x)$ and $(k', x')$ are two centered $\tau$-representations in $X$ of points $y, y' \in Y$ and $y_{[-R, R)} = y'_{[-R, R)}$, then $k = k'$ and $x_0 = x'_0$.

Arguing by contradiction, if the claim is not true and since $y_{(0, +\infty)} = y'_{(0, +\infty)}$ there exists $\ell_0 \ge 0$ such that if $\ell \ge \ell_0$ then $C_\tau^\ell(k, x) \notin \mathcal{C}_\tau^+(k', x')$ and $(S^j y)_{[-R, R)} = (S^j y')_{[-R, R)}$, where $j = C_\tau^{\ell_0}(k, x)$.
But then $j$ belongs to $\mathcal{C}_\tau^+(k', x')$, a contradiction.

By the claim, there exists an increasing sequence $(\ell_n)_{n \ge 0}$ such that $\ell_n \ge 0$ and
\[
C_\tau^{\ell_n}(k, x)
= C_\tau^{\ell'_n}(k', x'), \quad
\text{for some} \quad
\ell'_n \ge 0.
\]
If $\ell_n \ge 2$, since $y_{(0, +\infty)} = y'_{(0, +\infty)}$ and $C_\tau^{\ell_n}(k, x) = C_\tau^{\ell'_n}(k', x')$, we deduce that $\tau(x_{\ell_n - 1})$ is a suffix of $\tau(x'_{\ell'_n - 1})$ or that $\tau(x'_{\ell_n' - 1})$ is a suffix of $\tau(x_{\ell_n - 1})$.
By assumption, this implies that $x_{\ell_n - 1} = x'_{\ell'_n - 1}$.
By repeating the argument, we see that $\ell_n = \ell'_n$ and $C_\tau^\ell(k, x) = C_\tau^\ell(k', x')$, $1 \le \ell \le \ell_n$, $n \ge 0$.
Therefore, as $(\ell_n)_{n \ge 0}$ is increasing, we deduce that $\mathcal{C}_\tau^+(k, x) = \mathcal{C}_\tau^+(k', x')$, $x_0 \not= x'_0$ since $y_0 \not= y'_0$ and $x_{(0, +\infty)} = x'_{(0, +\infty)}$.
This completes the proof.
\end{proof}

Let $(X_\WW, S)$ be a minimal Ferenczi subshift and $\btau_\WW = (\tau_n : \AA_{n+1}^\ast \to \AA_n^\ast)_{n \ge 0}$ be the directive sequence given by \eqref{eq:tau}.
In order to study the asymptotic classes of $(X_\WW, S)$ we need the following definitions.

Define the words $L_n = a_{n-1,1} a_{n-1,2} \ldots a_{n-1,q_{n-1}-1}$ and $R_n = a_{n-1,0}$.
Observe that they satisfy
\[
\tau_n(a)
= L_n a R_n,
\quad a \in \AA_{n+1}, \quad n \ge 1.
\]
Inductively, define $L_{1,1} = L_1$, $R_{1,1} = R_1$, and for $n \ge 1$ we let
\begin{equation}\label{eq:tau1}
L_{1,n+1} = \tau_{[1,n+1)}(L_{n+1}) L_{1,n}
\quad \text{and} \quad
R_{1,n+1} = R_{1,n} \tau_{[1,n+1)}(R_{n+1}).
\end{equation}
Hence, we have
\[
\tau_{[1,n+1)}(a)
= L_{1,n} a R_{1,n}, \quad
a \in \AA_{n+1}, \quad
n \ge 1.
\]

\begin{proposition}\label{p:asymptotic_class}
A minimal Ferenczi subshift has a unique asymptotic class.
\end{proposition}

\begin{proof}
Let $(X_\WW, S)$ be a minimal Ferenczi subshift generated by the directive sequence $\btau_\WW = (\tau_n : \AA_{n+1}^\ast \to \AA_n^\ast)_{n \ge 0}$ given by \eqref{eq:tau}.
For $n \ge 0$ recall the definition of the subshift $X_{\btau_\WW}^{(n)}$ given in \Cref{ss:s-adic} and that $\AA_1$ is the set of values of the sequence $(a_{n,i} : n \ge 0,\ 0 \le i < q_n)$.

For $n \ge 1$ and $a \not= b$ in $\AA_n$, we have that the word $\tau_{[0,n)}(a)$ is not a suffix of the word $\tau_{[0,n)}(b)$.
Indeed, this is clear if $n = 1$.
If $n \ge 1$, by \eqref{eq:tau1} we have
\begin{equation}\label{eq:images_tau}
\tau_{[0,n+1)}(c)
= \tau_0(L_{1,n}) 0 1^c \tau_0(R_{1,n}), \quad
c \in \AA_{n+1},
\end{equation}
from which the claim follows easily.

Since $(X_\WW, S)$ is minimal and aperiodic, there exists at least one asymptotic class.
Let $z$ and $z'$ be two points in this asymptotic class such that $z_0 \not= z'_0$ and $z_{(0,+\infty)} = z'_{(0,+\infty)}$.
Without lost of generality, we assume that $z_0 = 0$ and $z'_0 = 1$.
By \Cref{l:recognizability}, there exist pairs $(k, y)$ and $(k', y')$ with $y, y' \in X_{\btau_\WW}^{(1)}$, $0 \le k < |\tau_0(y_0)|$, $0 \le k' < |\tau_0(y'_0)|$ and
\[
z = S^k \tau_0(y), \quad
z' = S^{k'} \tau_0(y').
\]

Since $z_0 = 0$ and $\tau_0(c) = 0 1^c$ for $c \in \AA_1$, we deduce that $k = 0$ and from \Cref{l:asymptotic} we obtain $C_{\tau_0}^1(0, y) = C_{\tau_0}^1(k', y')$.
Define $a = y_0$ and $b = y'_0$.
The fact that $z_{(0,+\infty)} = z'_{(0,+\infty)}$ and $z'_0 = 1$ implies that $a < b$.

Now fix $n \ge 0$.
By \Cref{l:recognizability}, there exist pairs $(j, x)$ and $(j', x')$ with $x, x' \in X_{\btau_\WW}^{(n+1)}$, $0 \le j < |\tau_{[0,n+1)}(x_0)|$, $0 \le j' < |\tau_{[0,n+1)}(x'_0)|$ and
\[
z = S^j \tau_{[0,n+1)}(x), \quad
z' = S^{j'} \tau_{[0,n+1)}(x').
\]
From \Cref{l:asymptotic} we have $C_{\tau_{[0,n+1)}}^1(j, x) = C_{\tau_{[0,n+1)}}^1(j', x')$.
Let $s = C_{\tau_{[0,n+1)}}^1(j, x)$, so that $z_{[1, s)} = 1^a \tau_0(y_{[1, m)}) = z'_{[1, s)}$ for some $m \in \NN$.
This, together with \eqref{eq:images_tau}, implies that $x_0 = a$ and $x'_0 = b$.
We conclude that
\[
z_{[1, s)}
= z'_{[1, s)}
= 1^a \tau_0(R_{1,n}).
\]

Since $R_{1,n}$ is a prefix of $R_{1,n+1}$ for each $n \ge 1$ and $(|R_{1,n}|)_{n \ge 1}$ is increasing, there exists a one-sided sequence $u = (u_n)_{n \in \NN}$ in $\{0, 1\}^\NN$ such that
\[
u_{[0, |\tau_0(R_{1,n})|)}
= \tau_0(R_{1,n}), \quad
n \ge 1.
\]
We deduce that $z_{[a+1, +\infty)} = z'_{[a+1, +\infty)} = u$, which does not depend on the points $z$ and $z'$ but only on $\btau_\WW$.
This proves the result.
\end{proof}

For a minimal topological dynamical system $(X, T)$, the existence of a unique asymptotic class implies that the automorphism group $\Aut(X, T)$ is trivial \emph{i.e.},
\[
\Aut(X, T)
= \langle T \rangle.
\]

Indeed, let $x \in X$ be an element in the unique asymptotic class and $\phi$ be an element in $\Aut(X, T)$.
Since the map $\phi$ sends asymptotic classes to asymptotic classes, we deduce that $\phi(x)$ is asymptotic to $T^m x$ for some $m \in \ZZ$.
From \cite[Lemma 2.3]{DDMP16} we have that $\phi = T^m$, and we conclude that $\Aut(X, T) = \langle T \rangle$.

Therefore, \Cref{p:asymptotic_class} implies the following.

\begin{corollary}[{\cite[Theorem 1.2]{GH16b}}]
The automorphism group of a minimal Ferenczi subshift is trivial.
\end{corollary}

\section{Measurable eigenvalues of minimal Ferenczi subshifts}
\label{s:measurable}

In this section we further develop the spectral study of minimal Ferenczi subshifts initiated in \Cref{ss:continuous_eigs} for continuous eigenvalues by analyzing their measurable eigenvalues.

We first give a general necessary condition for a complex number to be a measurable eigenvalue of certain $\SS$-adic subshifts.
This is stated in \Cref{p:veech_criterion}.
Then, we show that, under the hypothesis of exact finite rank, all measurable eigenvalues of minimal Ferenczi subshifts are continuous, thus improving previous known results \cite[Theorem 4.1]{GH16a}.
This is stated in \Cref{c:cor_veech}.

\subsection{The Veech criterion for \SSS-adic subshifts}

We now give a general necessary condition for a complex value to be a measurable eigenvalue with respect to an ergodic invariant probability measure of some $\SS$-adic subshifts.
Such a condition, originally due to W. Veech \cite{Vee84} in the context of interval exchange transformations, was stated as the \emph{Veech criterion} in several articles \cite{Vee84, AF07, AD16} and was crucial in order to obtain generic weak mixing for interval exchange transformations and translation flows in certain Veech surfaces.

For convenience, we state and prove here the necessary condition in the context of $\SS$-adic subshifts following the lines of the original proof of the Veech criterion.
Since we will only consider minimal Cantor systems of finite topological rank, there is no loss in generality \cite{DM08}.
See \cite{DFM19} for a finer analysis of measurable eigenvalues in the more general context of minimal Cantor systems.

Suppose that $\btau = (\tau_n : \AA_{n+1}^\ast \to \AA_n^\ast)_{n \ge 0}$ is a clean directive sequence, $\mu$ be an invariant ergodic probability measure of $(X_{\btau}, S)$ and let $\AA_\mu$ be as given in \eqref{eq:clean_ineq}.

\begin{proposition}\label{p:veech_criterion}
Let $\btau = (\tau_n : \AA_{n+1}^\ast \to \AA_n^\ast)_{n \ge 0}$ be a primitive, proper and recognizable directive sequence and $\mu$ be an ergodic invariant probability measure of $(X_{\btau}, S)$.
Assume that $\btau$ is clean with respect to $\mu$ and let $\AA_\mu$ be the set of letters such that \eqref{eq:clean_ineq} holds.
Suppose that:
\begin{enumerate}[label = (\roman*)]
	\item \label{item:veech1} there exists $K > 0$ such that $|\tau_{[0,n)}| / \langle \tau_{[0,n)} \rangle \le K$ for all large enough $n$; and
	
	\item \label{item:veech2} there exists $\delta > 0$, and, for all large enough $n$, a nonempty word $u_n \in \AA_0^\ast$ and indices $c_n, d_n$ with $0 \le c_n < d_n \le \min_{a \in \AA_\mu} |\tau_{[0,n)}(a)|$ which satisfy $|u_n| \ge \delta \min_{a \in \AA_\mu} |\tau_{[0,n)}(a)|$ and
	\[
	\tau_{[0,n)}(a)_{[c_n, d_n)} = u_n, \quad
	a \in \AA_\mu.
	\]
\end{enumerate}

If $\lambda = \exp(2 \pi i \alpha)$ is a measurable eigenvalue of $(X_{\btau}, S)$ with respect to $\mu$, then
\begin{equation}\label{eq:veech_criterion}
\lim_{n \to +\infty} \inttt{\alpha h_n(a)}
= 0, \quad 
a \in \AA_\mu.
\end{equation}
\end{proposition}

\begin{proof}
Let $f : X_{\btau} \to \CC$, $f \not= 0$ be a measurable eigenfunction of $(X_{\btau}, S)$ with respect to $\mu$ with eigenvalue $\lambda$.
We can assume $|f| = 1$ $\mu$-almost everywhere by ergodicity.
Remember the definition of the sets $B_n(a)$ for $a \in \AA_n$ and $B_n$ in \eqref{eq:rep_sadic}.

\medskip

Let $n_0 \in \NN$ and $c > 0$ be such that \eqref{eq:clean_ineq} holds and $0 < \epsilon < \frac{c}{3 K}$.
From now on, we choose $n \ge n_0$ large enough such that \Cref{item:veech1} and \Cref{item:veech2} hold.
We set
\[
B_{n, \mu}
= \bigcup_{a \in \AA_\mu} B_n(a).
\]
Define $t_n = c_n + \lceil (d_n - c_n) / 4 \rceil$, $\ell_n = \lceil (d_n - c_n) / 2 \rceil$ and $A_{n, \mu} = \bigcup_{k=t_n}^{t_n + \ell_n - 1} S^k B_{n, \mu}$.
See \Cref{f:veech_rectangles}.

\begin{figure}[ht]
	\centering
	\includegraphics{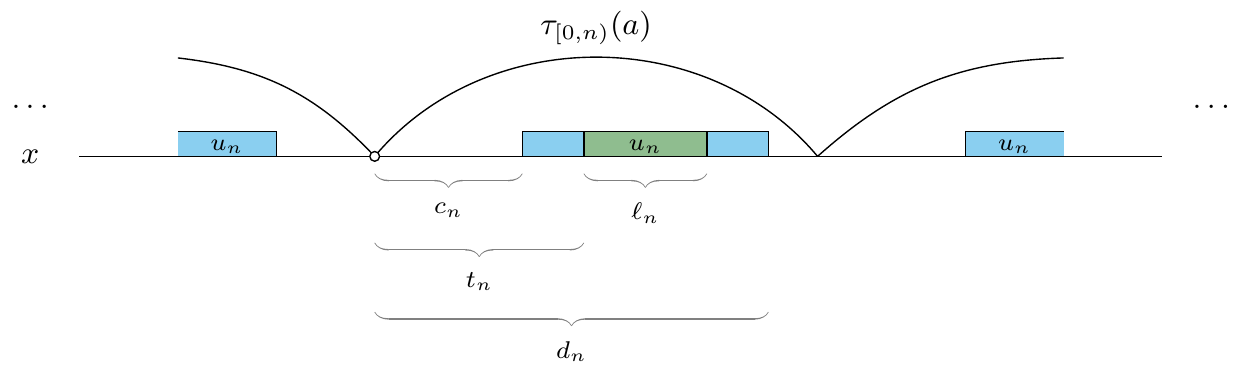}
	\caption{A centered $\tau_{[0,n)}$-representation of a point $x$ in $B_n(a)$, $a \in \AA_\mu$.
	The white point represents the zero coordinate of $x$.
	The green part represents the word ${u_n}_{[t_n, t_n + \ell_n)}$.
	}
	\label{f:veech_rectangles}
\end{figure}

Observe that the union which defines $A_{n, \mu}$ is disjoint and that
\begin{align*}
\mu(A_{n, \mu})
&\ge (d_n - c_n) \mu(B_{n, \mu}) / 2
\ge |u_n| \mu(B_{n, \mu}) / 2\\
&\ge \delta \langle \tau_{[0,n)} \rangle_\mu \mu(B_{n,\mu}) / 2\\
&\ge \delta c / 2.
\end{align*}

\begin{claim}
For all large enough $n$ there exists some value $k_n \in \NN$ and some complex value $w_{k_n} \in \CC$ such that $t_n \le k_n < t_n + \ell_n$ and
\[
\int_{S^{k_n} B_{n,\mu}} |f - w_{k_n}| d\mu
\le \epsilon^2 \mu(B_{n,\mu}).
\]
\end{claim}

Indeed, by Lusin's theorem there exists a compact set $C \subseteq X_{\btau}$ such that $f\vert_{C}$ is uniformly continuous and $\mu(C) \ge 1 - \chi$, where $\chi = \epsilon^2 \delta c / 8$.

Choose $k_n$ such that $\mu(S^{k_n} B_{n,\mu} \cap C) = \max_{t_n \le k < t_n + \ell_n} \mu(S^k B_{n,\mu} \cap C)$.
Then
\[
\dfrac{\mu(S^{k_n} B_{n,\mu} \cap C)}{\mu(B_{n,\mu})}
\ge \dfrac{\sum_{k=t_n}^{t_n + \ell_n - 1}
\mu(S^k B_{n,\mu} \cap C)}{\ell_n \mu(B_{n,\mu})}
= \dfrac{\mu(A_{n,\mu} \cap C)}{\mu(A_{n,\mu})}
\ge \dfrac{\mu(A_{n,\mu}) - \chi}{\mu(A_{n,\mu})},
\]
so that
\[
\mu(S^{k_n} B_{n,\mu} \cap (X_{\btau} \setminus C)) / \mu(B_{n,\mu})
\le \chi / \mu(A_{n,\mu})
\le \epsilon^2 / 4.
\]

By the choice of $t_n$, $\ell_n$ and since $|u_n| \to +\infty$, we have $\diam(S^{k_n} B_{n,\mu}) \to 0$ as $n \to +\infty$.
Hence, for all large enough $n$ we have
\[
\sup_{x, y \in S^{k_n} B_{n,\mu} \cap C} |f(x) - f(y)|
< \epsilon^2 / 2.
\]
Put $w_{k_n} = f(y_n)$ for some point $y_n \in S^{k_n} B_{n,\mu} \cap C$, then
\begin{align*}
\displaystyle\int_{S^{k_n} B_{n,\mu}} |f - w_{k_n}| d\mu
&= \displaystyle\int_{S^{k_n} B_{n,\mu} \cap C} |f - w_{k_n}| d\mu
+ \displaystyle\int_{S^{k_n} B_{n,\mu} \cap (X_{\btau} \setminus C)} |f - w_{k_n}| d\mu\\
&\le \frac{\epsilon^2}{2}\mu(B_{n,\mu}) + 2\mu(S^{k_n} B_{n,\mu} \cap (X_{\btau} \setminus C))\\
&\le \epsilon^2\mu(B_{n,\mu}),
\end{align*}
which proves the claim.

\medskip

Put $w_n' = w_{k_n} \lambda^{-k_n}$.
Since $f$ is an eigenfunction with associated eigenvalue $\lambda$ and $\mu$ is $S$-invariant, we deduce that
\begin{equation}\label{eq:veech}
\displaystyle\int_{B_{n,\mu}} |f - w_n'| d\mu
= \displaystyle\int_{S^{k_n}B_{n,\mu}} |f - w_{k_n}| d\mu
\le \epsilon^2 \mu(B_{n,\mu}).
\end{equation}
The Markov inequality and \eqref{eq:veech} imply
\begin{equation}\label{eq:markov_ineq}
\mu(\{x \in B_{n,\mu} : |f(x) - w_n'| \ge \epsilon\})
\le \epsilon \mu(B_{n,\mu}).
\end{equation}

\begin{claim}
Let $a \in \AA_\mu$.
Then, for all large enough $n$ there exists $x \in B_n(a)$ such that
\[
|f(x)| = 1, \quad
|f(x) - w_n'| < \epsilon \quad
\text{and} \quad
|f(S^{h_n(a)} x) - w_n'| < \epsilon.
\]
\end{claim}

Let $a \in \AA_\mu$.
We begin by observing that
\[
\dfrac{\mu(B_n \setminus B_{n,\mu})}{\mu(B_{n,\mu})}
= \dfrac{\sum_{b \in \AA \setminus \AA_\mu} \mu(\TT_n(b)) / |\tau_{[0,n)}(b)|}{\sum_{a \in \AA_\mu} \mu(\TT_n(a)) / |\tau_{[0,n)}(a)|}
\le \Big(\frac{K}{c |\AA_\mu|}\Big) \sum_{b \in \AA \setminus \AA_\mu} \mu(\TT_n(b)).
\]
Hence, by \eqref{eq:clean_ineq} for all large enough $n$ we obtain
\begin{equation}\label{eq:aux_ineq}
\mu(B_n \setminus B_{n,\mu})
< \epsilon \mu(B_{n,\mu}).
\end{equation}
Let $n$ be large enough so that \eqref{eq:markov_ineq} and \eqref{eq:aux_ineq} hold.
If the claim is not true for such $n$, after neglecting a set of measure zero we have

\begin{align*}
B_n(a)
\subseteq \{x \in B_{n,\mu} : |f(x) - w_n'| \ge \epsilon\} &\cup S^{-h_n(a)} \{x \in B_{n,\mu} : |f(x) - w_n'| \ge \epsilon\}\\
&\cup S^{-h_n(a)} \{x \in B_n \setminus B_{n,\mu} : |f(x) - w_n'| \ge \epsilon\},
\end{align*}
and then
\begin{equation}\label{eq:markov}
\mu_n(a)
< 3 \epsilon \mu(B_{n,\mu})
< \frac{c}{K} \mu(B_{n,\mu}).
\end{equation}

But by \eqref{eq:markov} we obtain the following contradiction
\begin{align*}
\mu(B_{n,\mu})
&= \sum_{b \in \AA_\mu} \mu_n(b)
= \sum_{b \in \AA_\mu} \mu_n(b)
\dfrac{|\tau_{[0,n)}(b)| |\tau_{[0,n)}(a)|}
{|\tau_{[0,n)}(b)| |\tau_{[0,n)}(a)|}\\
&\le \frac{K}{|\tau_{[0,n)}(a)|}
\sum_{b \in \AA_\mu} \mu(\TT_n(b))
\le \frac{K}{c} \mu_n(a) < \mu(B_{n,\mu}),
\end{align*}
where we used $\mu(\TT_n(a)) = \mu_n(a) |\tau_{[0,n)}(a)| \ge c$.
This proves the claim.

\medskip

Finally, we obtain $|\lambda^{h_n(a)} - 1| = |f(S^{h_n(a)} x) - f(x)| < 2 \epsilon$ and $\inttt{\alpha h_n(a)} \to 0$ as $n \to +\infty$ since $\epsilon$ can be chosen to be arbitrarily small.
\end{proof}

\begin{remark}
\noindent
\begin{enumerate}
    \item \Cref{item:veech2} in \Cref{p:veech_criterion} holds in particular if for each $n \ge 0$ and $a \in \AA$ there exists a prefix $p_n$ (or suffix $s_n$) of $\tau_{[0,n)}(a)$ and $\delta > 0$ such that the length $|p_n|$ (or $|s_n|$) is at least $\delta \langle \tau_{[0,n)} \rangle$.
    \item In \cite[Example 2]{DFM15} the authors describe an $\SS$-adic subshift of Toeplitz type and of exact finite rank such that $\exp(2 \pi i / 6)$ is a measurable and noncontinuous eigenvalue for the unique invariant probability measure.
    The associated height vectors $(h_n)_{n \ge 0}$ satisfy
    \[
    h_n(a)
    \equiv 1 \pmod{6}, \quad
    a \in \AA, \quad
    n \ge 0.
    \]
    A simple computation shows that \Cref{item:veech2} in \Cref{p:veech_criterion} does not hold.
    Therefore, the condition given by \eqref{eq:veech_criterion} is not always necessary.
\end{enumerate}
\end{remark}

\subsection{Veech criterion applied to Ferenczi subshifts}

We will apply \Cref{p:veech_criterion} to the study of measurable eigenvalues of minimal Ferenczi subshifts.
We first need the following lemma to fulfill \Cref{item:veech2}.

\begin{lemma}\label{l:one_third}
Let $(X_\WW, S)$ be a minimal Ferenczi subshift and $\btau_\WW = (\tau_n : \AA_{n+1}^\ast \to \AA_n^\ast)_{n \ge 0}$ be the directive sequence given in \eqref{eq:tau}.
Then, for all $n \ge 1$ there exists a common prefix $p_n$ of the words $\tau_{[0,n)}(a)$,  $a \in \AA_n$, satisfying
\[
|p_n|
\ge \frac{\min_{1\le i \le \ell} a_i + 1}{3(\max_{1\le i \le \ell} a_i + 1)} \langle \tau_{[0,n)}\rangle.
\]
\end{lemma}

\begin{proof}
Let $m = \min_{1 \le i \le \ell} a_i + 1$ and $M = \max_{1 \le i \le \ell} a_i + 1$.
Define $p_1 = 0 1^{m-1}$ and $p_n = \tau_0(L_{1,n-1})$, $n \ge 2$ as in \eqref{eq:tau1}.
By definition, $p_n$ is a prefix of $\tau_{[0,n)}(a)$ for each $a \in \AA_n$, $n \ge 1$.
Since each morphism $\tau_n$, $n \ge 1$ is of constant length, we have $|L_{1,n+1}| = |L_{1,n}| + |L_{n+1}| |\tau_{[1,n+1)}|$, $n \ge 1$.
Observe that $|L_n| = q_{n-1} - 1$ for $n \ge 1$, so that $|L_{1,1}| = q_0 - 1 \ge \frac{q_0 + 1}{3} = \frac{|\tau_{[1,2)}|}{3}$.

Inductively, if $|L_{1,n}| \ge \frac{|\tau_{[1,n+1)}|}{3}$, we obtain
\begin{align*}
|L_{1,n}| + |L_{n+1}| |\tau_{[1,n+1)}|
&\ge \frac{|\tau_{[1,n+1)}|}{3} + (q_n -1) |\tau_{[1,n+1)}|
\ge \frac{|\tau_{[1,n+1)}|}{3} (q_n + 1)\\
&= \frac{|\tau_{[1,n+2)}|}{3},
\end{align*}
where we used $q_n - 1 \ge \frac{q_n}{3}$.
This shows that $|L_{1,n+1}| \ge \frac{|\tau_{[1,n+2)}|}{3}$, $n \ge 0$.

Finally, we deduce $|p_1| = m \ge \frac{m}{3 M} \langle \tau_0 \rangle $ and
\[
|p_n|
\ge m |L_{1,n-1}|
\ge m \frac{|\tau_{[1,n)}|}{3}
\ge \frac{m}{3 M} \langle \tau_{[0,n)} \rangle, \quad n \ge 2.
\]
\end{proof}

For a minimal Ferenczi subshift $(X_\WW, S)$ with unique invariant probability measure $\mu$, we set $d_{\WW_\mu} = |\AA_\mu|$, where $\AA_\mu$ is defined as in \Cref{ss:clean}.
A direct application of \Cref{l:computation_heights}, \Cref{l:one_third} and \Cref{p:veech_criterion} allows us to obtain the following.

\begin{corollary}\label{c:cor_veech}
Let $(X_\WW, S)$ be a minimal Ferenczi subshift and $\mu$ be the unique invariant probability measure.
\begin{enumerate}
    \item If $d_{\WW_\mu} = d_\WW$ (\emph{i.e.}, if $\tau_\WW$ is of exact finite rank) then all measurable eigenvalues of $(X_\WW, S)$ with respect to $\mu$ are continuous.
    \item If $d_{\WW_\mu} \ge 2$, then the system $(X_\WW, S)$ has no irrational measurable eigenvalues with respect to $\mu$.
\end{enumerate}
\end{corollary}

\begin{proof}
Let $\lambda = \exp(2 \pi i \alpha)$ be a measurable eigenvalue of $(X_\WW, S)$ with respect to $\mu$ and $\btau_\WW$ be the directive sequence given by \eqref{eq:tau}.
We see that, by \Cref{l:one_third} and \Cref{l:computation_heights}, all hypothesis of \Cref{p:veech_criterion} are verified.
\begin{enumerate}
    \item If $d_{\WW_\mu} = d_\WW$, we have
    \[
    \inttt{\alpha h_n} \to 0
    \]
    as $n \to +\infty$.
    As in the proof of \Cref{p:rational_cont_eigs}, we have that $\alpha$ must be rational.
    Moreover, if $\alpha = p / q$ is rational, then $q$ must divide all coordinates of $h_n$ for all large enough $n$.
    We conclude that $\lambda$ is a continuous eigenvalue.
    \item If $d_{\WW_\mu} \ge 2$, there exist two distinct elements $a, b$ in $\AA_\WW$ which satisfy
    \[
    \inttt{\alpha(a - b)} = 0.
    \]
    In particular, $\alpha$ must be rational.
\end{enumerate}
\end{proof}

Thus, we showed that measurable and continuous eigenvalues coincide in the case where $d_{\WW_\mu} = d_\WW$.
However, when $d_{\WW_\mu} \not= d_\WW$, we can obtain different behaviors that we comment in the section below.

\subsection{Various examples exhibiting different spectral behaviors}

In this section we will precise the situation where $d_{\WW_\mu} \not= d_\WW$.
In this case, we give explicit examples of minimal Ferenczi subshifts $(X_\WW, S)$ having a prescribed topological rank $d_\WW$, a prescribed quantity $d_{\WW_\mu}$ with $d_{\WW_\mu} \not= d_\WW$, with no nontrivial continuous eigenvalue, but with any rational measurable eigenvalue $\lambda = \exp(2 \pi i /p)$.

However, when $d_{\WW_\mu} = 1$, we were not able to show that there are no irrational measurable eigenvalues.
We leave this as an open question.

\subsubsection{A realization result on measurable eigenvalues with rank constraints}

\begin{proposition}
Let $p$ be a prime number and $d, d'$ be such that $1 \le d' < d$.
Then, there exists a minimal Ferenczi subshift $(X_\WW, S)$ with unique invariant probability measure $\mu$ such that $\rank(X_\WW, S) = d$, $d_{\WW_\mu} = d'$, the system $(X_\WW, S)$ is topologically weakly-mixing and $\lambda = \exp(2 \pi i / p)$ is a measurable eigenvalue of $(X_\WW, S)$.
\end{proposition}

\begin{proof}
Consider any set of nonnegative numbers $\AA = \{a_i : 1 \le i \le d\}$ such that $p$ divides $a_i + 1$, $1 \le i \le d'$ and $a_d - a_{d'} = 1$.
Put $\AA_0 = \{0, 1\}$ and $\AA_n = \AA$ if $n \ge 1$.
Let $v = a_{j^\ast}$ for some $d' < j^\ast \le d$ and define the words
\begin{align*}
    U
    &= a_1 a_2 \ldots a_{d'}\\
    W
    &= a_{d' + 1} \ldots a_{j^\ast - 1} a_{j^\ast + 1} \ldots a_d
\end{align*}
in $\AA^\ast$.
Consider any increasing function $g : \NN \to \ZZ_{> 0}$ such that $\sum_{n=0}^\infty \frac{1}{g(n)} < +\infty$.
Let $\btau_\WW = (\tau_n : \AA_{n+1}^\ast \to \AA_n^\ast)_{n \ge 0}$ be the directive sequence given by
\begin{equation}\label{eq:morphisms_3}
    \tau_n(a)
    = U^{p g(n)} W^p v^{p-1} a v, \quad
    a \in \AA, \quad
    n \ge 1.
\end{equation}

\Cref{t:recognizable_sequence} implies that $\btau_\WW$ defines a minimal Ferenczi subshift $(X_\WW, S)$.
Moreover, \Cref{c:rank} implies that $\rank(X_\WW, S) = d$.
Let $(\TT_n)_{n \ge 0}$ be the nested sequence of CKR partitions of $X_\WW$ given by \eqref{eq:rep_sadic} and $\mu$ be the unique invariant probability measure of $(X_\WW, S)$.

Observe that the vector $(f_n(a))_{a \in \AA}$ associated with the morphism $\tau_n$, as defined by \eqref{eq:f_n}, is given by  
\[
f_n(a_i)
= \begin{cases}
pg(n) &\text{if \quad $1 \le i \le d'$}\\
p     &\text{if \quad $d' < i \le d$}.
\end{cases}
\]
From \eqref{eq:composition_matrices} the composition matrix of the morphism $\tau_n$ is given by
\[
M_{\tau_n}
= I + f_n \cdot \uu,
\]
where $I$ is the identity matrix indexed by $\AA$ and $\uu$ is the row vector of ones in $\RR^\AA$.

\begin{claim}
The system $(X_\WW, S)$ is topologically weakly-mixing and $d_{\WW_\mu} = d'$.
\end{claim}

\begin{proof}
Since $a_d - a_{d'} = 1$, we have that the system  $(X_\WW, S)$ is topologically weakly mixing.
 
It remains to show that $\AA_\mu = \{a_i : 1 \le i \le d'\}$.
From \eqref{eq:proportional_heights}, there exists a constant $K$ such that for all $a, b \in \AA$ and $n \ge 0$ we have
\begin{align*}
h_{n+1}(b)
&= \sum_{c \in \AA} h_n(c) M_{\tau_n}(c, b)
\le K h_n(a) \sum_{c \in \AA} M_{\tau_n}(c, b)
\le K |\AA| (p g(n) + 1) h_n(a),
\end{align*}
and 
\begin{align*}
h_{n+1}(b)
&= \sum_{c \in \AA} h_n(c) M_{\tau_n}(c, b)
\ge K^{-1} h_n(a) \sum_{c \in \AA} M_{\tau_n}(c, b)
\ge K^{-1} |\AA| p g(n) h_n(a).
\end{align*}
Let $i \in \{1, \ldots,  d'\}$, $j \in \{ d'+1, \ldots, d\}$, $b \in \AA$ and $n \ge 0$.
From above, we obtain
\begin{align*}
\mu(\TT_n(a_i))
&= h_n(a_i) \mu_n(a_i)
= h_n(a_i) \sum_{b \in \AA} M_{\tau_n}(a_i, b) \mu_{n+1}(b)\\
&\ge p g(n) h_n(a_i) \sum_{b \in \AA} \frac{\mu(\TT_{n+1}(b))}{h_{n+1}(b)} \ge p g(n) \min_{b \in \AA} \frac{h_n(a_i)}{h_{n+1}(b)}\\
&\ge \frac{p g(n)}{K |\AA| (p g(n) + 1)}
\ge \frac{1}{2 K |\AA|}
\end{align*}
and
\begin{align*}
\mu(\TT_n(a_j))
&= h_n(a_j) \mu_n(a_j)
= h_n(a_j) \sum_{b \in \AA} M_{\tau_n}(a_j, b) \mu_{n+1}(b)\\
&\le (p+1) h_n(a_j) \sum_{b \in \AA} \frac{\mu(\TT_{n+1}(b))}{h_{n+1}(b)}\le (p+1) \max_{b \in \AA} \frac{h_n(a_j)}{h_{n+1}(b)}\\
&\le \frac{K(p+1)}{|\AA| p g(n)}
\to 0
\end{align*}
as $n \to +\infty$.
This shows that $\AA_\mu = \{a_i : 1 \le i \le d'\}$ and proves the claim.
\end{proof}

\medskip

Let $\lambda = \exp(2 \pi i / p)$.
We will define a measurable eigenfunction $f$ as a limit of the sequence $(f_n)_{n \ge 0}$ defined by
\[
f_n(x) = \exp(2 \pi i j / p) \quad
\text{if} \quad
x \in S^j B_n(a), \quad
0 \le j < h_n(a), \quad
a \in \AA.
\]
Define
\[
A_n
= \{x \in X_\WW : f_{n+1}(x) \not= f_n(x)\}, \quad n \ge 0
\]
and
\[
t_n
= |\tau_{[0,n)}(U^{p g(n+1)})|, \quad
n \ge 0.
\]
Observe, from \eqref{eq:composition_matrices}, that $p$ divides $h_n(a)$ for all $a \in \AA_\mu$.
Moreover, $f_{n+1}(x) = f_n(x)$ whenever $x$ belongs to
\[
\bigcup_{1 \le i \le d'} \bigcup_{0 \le j < t_n} S^j B_{n+1}(a_i).
\]
Consequently, we deduce
\[
A_n
\subseteq \left( \bigcup_{1 \le i \le d'} \bigcup_{t_n \le j < h_{n+1}(a_i)} S^j B_{n+1}(a_i) \right) \cup \left( \bigcup_{d' < j \le d} \TT_{n+1}(a_j) \right).
\]

Let $1 \le i \le d'$.
From \eqref{eq:proportional_heights}, we have
\[
h_{n+1}(a_i) - t_n
= h_n(a_i) + p \sum_{d' < j \le d} h_n(a_j)
\le p K |\AA| h_n(a_i).
\]

From the previous computations, we obtain
\begin{align*}
\mu(A_n)
&\le p K |\AA| \sum_{1 \le i \le d'} \mu_{n+1}(a_i) h_n(a_i)
+ \sum_{d' < j \le d} \mu(\TT_{n+1}(a_j))\\
&\le p K |\AA| \sum_{1 \le i \le d'} \frac{h_n(a_i)}{h_{n+1}(a_i)} \mu(\TT_{n+1}(a_i))
+ \frac{K(p+1)}{p g(n+1)}\\
&\le \frac{K^2}{g(n)}
+ \frac{K(p+1)}{p g(n+1)} \le \frac{2 K^2 (p+1)}{g(n)}
\end{align*}
and consequently $\sum \mu(A_n)$ converges.
The Borel--Cantelli lemma implies that $\mu(\limsup_{n \to +\infty} A_n) = 0$.
Hence, the sequence $(f_n)_{n \ge 0}$ converges $\mu$-almost everywhere to some function $f$.

Moreover, if $x$ is not in $\bigcup_{a \in \AA} S^{h_n(a) - 1} B_n(a)$, then $f_n(S x) = \lambda f_n(x)$.
Since $\mu(\bigcup_{a \in \AA} S^{h_n(a) - 1} B_n(a)) \to 0$ as $n \to +\infty$, we conclude that $f$ is a measurable eigenfunction with eigenvalue $\lambda$ of $(X_\WW, S)$ with respect to $\mu$.
\end{proof}

\subsubsection{An example with \texorpdfstring{$d_{\WW_\mu} = 1$}{} with no noncontinuous rational eigenvalue}

We now present an example in the situation where $d_{\WW_\mu} = 1$ and every measurable rational eigenvalue is continuous.
For this purpose, we will use the following useful result of \cite[Section 4]{DFM19}.
Note that we have adapted this result to fit the context of $\SS$-adic subshifts.

\begin{lemma}[{\cite[Corollary 16]{DFM19}}]\label{l:cor16}
Let $\btau = (\tau_n : \AA_{n+1}^\ast \to \AA_n^\ast)_{n \ge 0}$ be a proper, primitive and recognizable directive sequence and let $\mu$ be an ergodic invariant probability measure of $(X_{\btau}, S)$.
Assume that $\btau$ is clean with respect to $\mu$ and let $\AA_\mu$ be the set of letters such that \eqref{eq:clean_ineq} holds.

Let $\lambda$ be a complex number of modulus $1$.
If for all $a, b \in \AA_\mu$,
\begin{equation}\label{eq:cor16}
\dfrac{\Big | \sum_{w \in W_{m, n}(a, b)} \lambda^{\langle \ell(w), h_m \rangle} \Big |}
{|\tau_{[m, n)}(b)|_a}
\to 1 \quad
\text{as} \quad
m \to +\infty \quad
\text{uniformly for} \quad
n > m,
\end{equation}
then $\lambda$ is an eigenvalue of $(X_{\btau}, S)$ with respect to $\mu$, where
\begin{itemize}
    \item $h_m = (|\tau_{[0,m)}(a)| : a \in \AA_m)$ for $m > 0$ 
    \item the set $W_{m, n}(a, b)$ is defined by
    \[
    \{
    \tau_{[m, n)}(b)_{[i,|\tau_{[m, n)}(b)|)}:
    \text{$i$ occurrence of $a$ in $\tau_{[m, n)}(b)$}
    \}
    \]
    \item for a word $w \in \AA_m^\ast$,
    \[
    \ell(w)
    = (|w|_a : a \in \AA_m).
    \]
\end{itemize}
The converse is also true, up to a contraction of the directive sequence $\btau$.
\end{lemma}

Let us precise here that a sequence $(a_{m,n})_{m,n \ge 0}$ converges to $\ell$ as $m \to +\infty$ uniformly for $n > m$ if for every $\epsilon > 0$ there exists $m_0 \ge 0$ such that for all $n > m \ge m_0$ we have
\[
|a_{m,n} - \ell| < \epsilon.
\]

Let $a,b$ be two positive integers with $a > b$ and let $\btau_\WW = (\tau_n : \AA_{n+1}^\ast \to \AA_n^\ast)_{n \ge 0}$ be the directive sequence given by
\begin{equation}\label{eq:morphisms_1}
\tau_n(a) = a^n b a^n \quad
\text{and} \quad
\tau_n(b) = a^n b a^{n-2} b a, \quad
n \ge 1.
\end{equation}
\Cref{t:recognizable_sequence} implies that $\btau_\WW$ defines a minimal Ferenczi subshift $(X_\WW, S)$.

\begin{proposition}
The minimal Ferenczi subshift $(X_\WW , S)$ defined by \eqref{eq:morphisms_1} is such that $d_{\WW_\mu} = 1$ and every rational measurable eigenvalue is continuous, for the unique ergodic measure $\mu$ of $(X_\WW , S)$.
\end{proposition}

\begin{proof}
We say that a word $w \in \{a,b\}^\ast$ consists of \emph{$a$-blocks} if we can write
\begin{equation}\label{eq:blocks}
w
= B_0(w) b B_1(w) b \ldots b B_{b(w) - 1}(w),
\end{equation}
where $b(w) \ge 2$ and $B_j (w)$ is a nontrivial power of $a$ for $0 \le j < b(w)$.

Define $w(m,n) = \tau_{[m,n)}(a)$ and $w'(m,n) = \tau_{[m,n)}(b)$ for $n > m \ge 0$.
It is easy to check that $w(m,n)$ and $w'(m,n)$ consist of $a$-blocks.
We will use the notation introduced in \eqref{eq:blocks} for these words.

We have $\AA_\mu = \{a\}$.
Indeed, the composition matrix $M_{\tau_n}$ of the morphism $\tau_n$ is
\[
M_{\tau_n}
= \matriz{2n & 2n-1 \\ 1 & 2}.
\]
Hence, from \eqref{eq:prob_measure} and \eqref{eq:inv_measure}, we obtain
\begin{align*}
\mu(\TT_n(b))
&= h_n(b) \mu_n(b)
= h_n(b)
(\mu_{n+1}(a) + 2 \mu_{n+1}(b))\\
&= h_n(b)
\Bigg(
\frac{\mu(\TT_{n+1}(a))}{h_{n+1}(a)} + 2 \frac{\mu(\TT_{n+1}(b))}{h_{n+1}(b)}
\Bigg)
\le \frac{2}{2n+1},
\end{align*}
where we used $h_{n+1}(a) = 2n h_n(a) + h_n(b) \ge (2n+1) h_n(b)$ by \eqref{eq:height_vectors} and, analogously, $h_{n+1}(b) \ge (2n+1) h_n(b)$.
Thus $d_{\WW_\mu} = 1$ and, from \cite[Proposition 5.1]{BKMS13}, we deduce
\begin{equation}\label{eq:occurence_a}
|w(m,n)|_a / |w(m,n)|
\to 1 \quad
\text{as} \quad
m \to +\infty \quad
\text{uniformly for} \quad
n > m.
\end{equation}

Suppose that $p$ is a prime number and that $\lambda = \exp(2 \pi i / p)$ is a rational eigenvalue of $(X_{\WW}, S)$ with respect to $\mu$ that is not continuous.
From \eqref{eq:cor16} and \eqref{eq:occurence_a}, we should have 
\begin{equation}\label{eq:cor16_contr}
\frac{
\Big |\sum_{w \in W_{m, m+2}(a, a)} \lambda^{\langle \ell(w), h_m \rangle}
\Big |
}
{|w(m,m+2)|}
\to 1 \quad
\text{as} \quad
m \to +\infty.
\end{equation}

Let us show this is not the case.
Let $m \ge 0$.
From the definition of the set $W_{m,m+2}(a, a)$ and \eqref{eq:blocks}, we have
\[
\sum_{w \in W_{m, m+2}(a, a)} \lambda^{\langle \ell(w), h_m \rangle}
= \sum_{j=0}^{b(w(m,m+2)) - 1}
\sum_{i=0}^{|B_j (w(m,m+2))| - 1}
\lambda^{\langle \ell(u_{j,i}), h_m \rangle},
\]
where $u_{j,i} = w(m,m+2)_{[
\left( \sum_{k = 0}^{j-1} |B_k(w(m,m+2))| \right) + j + i , |w(m,m+2))|}$.

On the other hand, from \eqref{eq:height_vectors} we have $h_m(b) = h_m(a) + (b-a)$, therefore
\begin{align*}
\langle
\ell(u_{j,i}) , h_m
\rangle
&= h_m(a) \Big (\sum_{k=0}^{j-1} |B_k (w(m,m+2))| + i \Big ) + h_m(b)j\\
&= h_m(a) \Big (\sum_{k=0}^{j-1} |B_k (w(m,m+2))| + i + j \Big ) + (b-a)j.
\end{align*}

Since $\AA_\mu = \{a\}$, the Veech criterion \eqref{eq:veech_criterion} implies that $p$ divides $h_m(a)$ for all large enough $m$.
Moreover, as $\lambda$ is a noncontinuous eigenvalue, $p$ does not divide $h_m(b)$ for any $m \ge 0$, and hence $p$ does not divide $b-a$.
Denote by $(b-a)^{-1}$ the inverse of $b-a \pmod{p}$.
For each $0 \le \ell < p$, let $r_\ell$ be such that $0 \le r_\ell < p$ and $r_\ell \equiv \ell \cdot (b - a)^{-1} \pmod{p}$.

For all large enough $m$ we have
\begin{align}\label{eq:cor16_blocks}
\sum_{w \in W_{m, m+2}(a, a)} \lambda^{\langle \ell(w), h_m \rangle}
& = \sum_{j=0}^{b(w(m,m+2)) - 1}
|B_j (w(m,m+2))| \exp(2 \pi i (b-a)j / p) \nonumber\\
& = \sum_{\ell = 0}^{p-1} a_\ell \exp(2 \pi i \ell / p),
\end{align}
with
\begin{equation}\label{eq:al}
a_\ell
= \sum_{j=0}^{\lfloor b(w(m,m+2)) / p \rfloor - 1} |B (w(m,m+2))_{jp + r_\ell}|.
\end{equation}

From the shape of the images of the morphism $\tau_m$ \eqref{eq:morphisms_1} and the fact that $\tau_{[m,m+2)}(a)$ belongs to the free monoid $\{\tau_m(a), \tau_m(b)\}^\ast$, we have
\begin{equation}\label{eq:comparable}
|B_j (w(m,m+2)|
- |B_{j'}v (w(m,m+2))|
\le m+2, \quad
0 \le j,j' < b(w(m,m+2)).
\end{equation}

This, together with \eqref{eq:al} implies that
\begin{equation}\label{eq:comparable_al}
a_\ell - \min_{0 \le \ell' < p} a_{\ell'}
\le
(m+2) \frac{b(w(m,m+2))}{p}, \quad
0 \le \ell < p.
\end{equation}

From \eqref{eq:cor16_blocks} and \eqref{eq:comparable_al}, we have
\begin{equation}\label{eq:ineq_comparable}
\frac{
\Big |\sum_{w \in W_{m, m+2}(a, a)} \lambda^{\langle \ell(w), h_m \rangle}
\Big |
}
{|w(m,m+2)|}
\le 
(m+2) \frac{b(w(m,m+2))}{|w(m,m+2)|},
\end{equation}
where we used the fact that $\sum_{\ell=0}^{p-1} \exp(2 \pi i \ell / p) = 0$.

\medskip

By computing $\tau_m \circ \tau_{m+1}(a)$, one can easily check that the number of $a$-blocks of $\tau_{[m,m+2)}(a)$ is $b(w(m,m+2)) = 2m + 5$.
Then we deduce
\[
(m+2) \frac{b(w(m, m+2))}{|w(m, m+2)|}
= \frac{(m+2)(2m+5)}{(2m+1)(2m+3)}
\to
\frac{1}{2} \quad
\text{as} \quad
m \to +\infty.
\]

This contradicts \eqref{eq:cor16_contr} and finishes the proof.
\end{proof}


\sloppy\printbibliography

\end{document}